\documentclass{amsart}
\usepackage[utf8]{inputenc}
\usepackage{amsmath}
\usepackage{amssymb}
\usepackage{mathrsfs}
\usepackage{graphicx}

\usepackage[colorlinks]{hyperref}
\usepackage[svgnames, dvipsnames, usenames]{xcolor}
\hypersetup{urlcolor = RedViolet, linkcolor = RoyalBlue, citecolor = ForestGreen}

\usepackage{partmac}
\usepackage{graphmor}

\def\ppen{\penalty300 }
\let\col=\colon
\def\colon{\col\ppen}

\theoremstyle{plain} 
\newtheorem{thm}{Theorem}[section]
\newtheorem{prop}[thm]{Proposition}
\newtheorem{lem}[thm]{Lemma}
\newtheorem{cor}[thm]{Corollary}
\theoremstyle{definition}
\newtheorem{defn}[thm]{Definition}
\newtheorem{nota}[thm]{Notation}
\newtheorem{rem}[thm]{Remark}
\newtheorem{ex}[thm]{Example}
\numberwithin{equation}{section}

\theoremstyle{plain}
\newtheorem{innercustomthm}{Theorem}
\newenvironment{customthm}[1]
  {\renewcommand\theinnercustomthm{#1}\innercustomthm}
  {\endinnercustomthm}

\renewcommand{\theta}{\vartheta}
\renewcommand{\phi}{\varphi}
\renewcommand{\epsilon}{\varepsilon}
\renewcommand{\subset}{\subseteq}
\renewcommand{\supset}{\supseteq}

\newcommand{\N}{\mathbb N}
\newcommand{\Z}{\mathbb Z}

\newcommand{\C}{\mathbb C}
\newcommand{\GGr}{\mathbb G}
\newcommand{\HGr}{\mathbb H}

\DeclareMathOperator{\Mor}{Mor}

\DeclareMathOperator{\spanlin}{span}
\DeclareMathOperator{\Lrot}{Lrot}
\DeclareMathOperator{\Rrot}{Rrot}

\DeclareMathOperator{\id}{id}
\DeclareMathOperator{\diag}{diag}
\DeclareMathOperator{\Aut}{Aut}
\DeclareMathOperator{\End}{End}
\newcommand{\Cat}{\mathscr{C}}
\newcommand{\RCat}{\mathfrak{C}}
\newcommand{\Gat}{\mathscr{G}}

\newcommand{\GG}{\mathbf{G}}
\newcommand{\KG}{\mathbf{K}}
\newcommand{\HG}{\mathbf{H}}
\newcommand{\MG}{\mathbf{M}}

\newcommand{\nullG}{\mathbf{0}}

\newcommand{\grpth}{^{\mathrm{grp}\mathchar `\-\mathrm{th}}}
\newcommand{\staralg}{\mathop{\rm\ast\mathchar `\-alg}}
\newcommand{\llangle}{\langle\!\langle}
\newcommand{\rrangle}{\rangle\!\rangle}

\newcommand{\PG}{\mathbf{P}}
\newcommand{\QG}{\mathbf{Q}}

\newcommand{\Lin}{\mathscr{L}}
\newcommand{\Part}{\mathscr{P}}

\newcommand{\aw}{\mathbf{a}}
\newcommand{\bw}{\mathbf{b}}
\newcommand{\cw}{\mathbf{c}}
\newcommand{\dw}{\mathbf{d}}
\newcommand{\iw}{\mathbf{i}}
\newcommand{\jw}{\mathbf{j}}
\newcommand{\vw}{\mathbf{v}}

\newcommand{\la}{\mathsf{a}}

\newcommand{\lv}{\mathsf{v}}

\newcommand{\Gedge}{\Graph{
\GV 0.5:1,2;
\GS 0.5/0.5:1/2;
}}

\newcommand{\Gtriangle}{\Graph{
\GV 0.2:1,2;
\GV 0.8:1.5;
\GS 0.2/0.2:1/2;
\GS 0.2/0.8:1/1.5,2/1.5;
}}

\def\Gforkk{\Graph{
\GS 0.3/0.5:1/1.5,1.5/,2/1.5;
\GS 0.5/1:1.5/;
\GS 0.3/0:1/,1.5/,2/;
\GV 0.5:1.5;
}}

\newcommand{\Gabab}{\Graph{
\GE 1/1:2/3;
\GS 1/0:2/1,2/3,3/2,3/4;
\GV 1:2,3;
}}

\newcommand{\GLabc}{\Graph{
\GV 0.5:1,2,3;
\GS 0.5/0:1/,2/,3/;
}}

\begin{document}
\title{Group-theoretical graph categories}
\author{Daniel Gromada}
\address{Saarland University, Fachbereich Mathematik, Postfach 151150,
66041 Saarbr\"ucken, Germany}
\email{gromada@math.uni-sb.de}
\date{\today}
\thanks{The author was supported by the collaborative research centre SFB-TRR~195 ``Symbolic Tools in Mathematics and their Application''.}
\thanks{I would like to thank Alexander Mang, David Roberson, and Moritz Weber for inspiring discussions about graph categories. I also thank to Laura Maaßen for discussing details of her work and the relations with this article.}
\thanks{I am grateful to Moritz Weber also for reading and commenting an early draft of this manuscript. Many thanks go to the referee for a detailed review that significantly improved the quality of the article.}

\begin{abstract}
The semidirect product of a~finitely generated group dual with the symmetric group can be described through so-called group-theoretical categories of partitions (covers only a~special case; due to Raum--Weber, 2015) and skew categories of partitions (more general; due to Maaßen, 2018). We generalize these results to the case of graph categories, which allows to replace the symmetric group by the group of automorphisms of some graph.
\end{abstract}

\maketitle
\section*{Introduction}

The main subject of this article are diagrammatic categories that can be used to model representation categories of groups and quantum groups. One typical example is the \emph{category of all partitions}~$\Part$, which models the intertwiner spaces of the symmetric group \cite{HR05}. Another classical example is so-called \emph{Brauer duality}, which connects the \emph{category of all pairings} with the orthogonal group. Some additional categories of partitions (i.e. subcategories of the category of all partitions) were later interpreted in the theory of \emph{quantum groups} \cite{BS09}. A~recent paper \cite{MR19} studies categories of graphs. In particular, the \emph{category of all graphs} can be used to model the representation category of the automorphism group $\Aut G$ of a~given graph~$G$.

These results can all be understood as some generalizations of the classical Schur--Weyl duality. But we can also go the opposite direction. Starting with a~certain category, we can reconstruct the unique compact matrix quantum group associated to this category using the so-called \emph{Tannaka--Krein reconstruction} (for quantum groups formulated by Woronowicz \cite{Wor88}). This idea motivated the search for classification of all categories of partitions as every such instance defines a~new compact matrix quantum group.

This classification was successfully completed in \cite{RW16}. An important large class of the categories of partitions is formed by so-called \emph{group-theoretical categories} \cite{RW14,RW15}. Those categories are shown to be in one-to-one correspondence with $sS_\infty$-invariant normal subgroups $A\trianglelefteq\Z_2^{*\infty}$. In addition, it is shown that the associated quantum group is of the form $\hat\Gamma\rtimes S_n$, where $\Gamma=\Z_2^{*n}/A$.

The semidirect product construction $\hat\Gamma\rtimes S_n$ makes sense also if $A$ is $S_\infty$-invariant but not $sS_\infty$-invariant. In that case, the standard categories of partitions cannot be used to describe the intertwiners. A~natural question is then, whether there is an alternative approach to model the representation categories. This was solved in \cite{Maa18} by introducing skew categories of partitions.

The goal of the current article is to generalize those concepts defined for categories of partitions into the more general context of categories of graphs. The article is divided into two parts -- the first part dealing with combinatorics and the second part dealing with the quantum groups and intertwiners. Each part consists of a~preliminary section (Sect. \ref{sec.prelimcomb} and~\ref{sec.prelimQG}) and sections with original results (Sect. \ref{sec.comb}, \ref{sec.QG},~\ref{sec.conv}).

In Section~\ref{sec.prelimcomb}, we introduce categories of partitions and graphs and recall the important combinatorial results from \cite{RW14} and \cite{Maa18}. In Section~\ref{sec.comb}, we generalize those results into the graph-categorical context. We define \emph{group-theoretical graph categories} to be those graph categories that are invariant under taking graph quotients. We also introduce more general \emph{skew graph categories}. We show that such categories can be described in a~group-theoretical manner in terms of \emph{graph fibrations}.

\begin{customthm}{A}[Theorem \ref{T.comb}]
There is a~one-to-one correspondence between graph fibrations~$F$ and skew graph categories~$\Cat$ described by Formulae \eqref{eq.CatF},~\eqref{eq.FCat}. The graph fibration~$F$ is easy if and only if the skew graph category~$\Cat$ is a~group-theoretical graph category.
\end{customthm}

In Section~\ref{sec.prelimQG}, we introduce compact matrix quantum groups and their connections with categories of partitions and graphs. In particular, we recall the important results from \cite{RW15,Maa18}. In Section~\ref{sec.QG}, we generalize those results into the graph-categorical context. First, we look at the \emph{easy} case. In Theorem~\ref{T.Geasy}, we show that group-theoretical graph categories correspond to quantum groups of the form $\hat\Gamma\rtimes\Aut G$. Then we focus on the more general case of skew graph categories and obtain the following result.
\begin{customthm}{B}[Theorem \ref{T.G}]
\label{T.B}
Let $F$ be a~graph fibration and $\Cat$~the corresponding category. Let $G$ be a~graph, let $K$ be the greatest subgraph of~$G$ contained in~$F$. Then the quantum group corresponding to~$\Cat^G$ is
$$\GGr=\hat\Gamma\rtimes\Aut K,$$
where $\Gamma=\Z_2^{*V(K)}/F(K)$.
\end{customthm}

Finally, in Section~\ref{sec.conv}, we discuss the opposite question: Given a~quantum group of the form $\GGr=\hat\Gamma\rtimes\Aut G$, what graph category can we associate to it and what are the corresponding intertwiner spaces. In Theorem~\ref{T.intsp}, we even generalize Theorem~\ref{T.B} and describe the intertwiner spaces for the semidirect product $\hat\Gamma\rtimes H$, where $H\subset S_n$ is an arbitrary permutation group and $\Gamma$ is an $H$-invariant quotient of $\Z_2^{*n}$.

\section{Preliminaries: Partitions, graphs, and their categories}
\label{sec.prelimcomb}

\subsection{Partitions}

Let $S$ be a~set. A~\emph{partition}~$\pi$ of~$S$ is a~decomposition of the set~$S$ into disjoint non-empty subsets. That is, $\pi=\{V_1,\dots,V_n\}$ with $V_i\neq\emptyset$, $V_i\cap V_j=\emptyset$, and $\bigcup_i V_i=S$. We denote $\pi\in\Part(S)$.

Consider $k\in\N_0$. By a~\emph{partition of $k$ points}, we mean a~partition of the set $\{1,\dots,k\}$. We denote $\Part(k):=\Part(\{1,\dots,k\})$ the set of all such partitions. We denote such partitions pictorially by drawing the $k$~points on a~line and connecting those that are contained in the same part by strings. For example,
$$p=\{\{1,2,4\},\{3,6\},\{5\},\{7,8\}\}=\LPartition{0.5:5}{0.5:1,2,4;1:3,6;0.5:7,8}\in\Part(8).$$

Another way to represent partitions is to use words. Let $V$ be some countable alphabet. We denote by~$V^*$ the monoid of words over~$V$. By~$V^k$, we denote all words of length~$k$. We denote by $\emptyset\in V^0$ the empty word. We may represent partitions of $k$~points as words in~$V^k$ by identifying the blocks of~$p$ with some letters in~$V$. Note that this representation is not unique as we may always choose a~different identification between letters and blocks. Given a~word $\aw\in V^k$, we denote by $\ker(\aw)$ the associated partition. For example, the above partition $p$ can be represented by the following words
$$p=\LPartition{0.5:5}{0.5:1,2,4;1:3,6;0.5:7,8}=\ker(\mathsf{aabacbdd})=\ker(\mathsf{ccecgeaa})$$

Consider $k,l\in\N_0$. We denote by $\Part(k,l)$ the set of \emph{partitions of $k$ upper and $l$~lower points}, that is, partitions of the set $\{1,\dots,k\}\sqcup\{1,\dots,l\}$. Those will again be represented pictorially. This time, we put $k$ points in one line and $l$ on another line below and connect again those points that share the same part. In order to make clear, whether crossing strings denote a~single block or two separate ones, we denote the blocks by dots. That is, all strings coming from points belonging to a~single block intersect in a~single point emphasized by a~dot. For example,
\begin{equation}
\label{eq.pq}
\PG=
\Graph{
\GV 0.5:2.5;
\GV 0:1,4;
\GS 0.5/1:2.5/1,2.5/2,2.5/3;
\GS 0.5/0:2.5/2,2.5/3;
\GS 1/1.5:1/,2/,3/;
\GS 0/-0.5:1/,2/,3/,4/;
}\in\Part(3,4)
\qquad
\QG=
\Graph{
\GV 1:1,4;
\GV 0:1;
\GV 0.25:2.25,3.5;
\GS 0.25/0:2.25/2,3.5/3,3.5/4;
\GS 0.25/1:2.25/3,3.5/2;
\GS 1/1.5:1/,2/,3/,4/;
\GS 0/-0.5:1/,2/,3/,4/;
}\in\Part(4,4).
\end{equation}

The partitions on two lines can also be represented by pairs of words. The first word represents the top line, the second word represents the bottom line.
$$
\vrule height 16bp depth 10bp width 0bp
\PG=
\Graph{
\GV 0.5:2.5;
\GV 0:1,4;
\GS 0.5/1:2.5/1,2.5/2,2.5/3;
\GS 0.5/0:2.5/2,2.5/3;
\GS 1/1.5:1/,2/,3/;
\GS 0/-0.5:1/,2/,3/,4/;
\Ptext (3.2,0.5) {$\mathsf{a}$}
\Ptext (0.5,0) {$\mathsf{b}$}
\Ptext (4.5,0) {$\mathsf{c}$}
}=\ker(\mathsf{aaa},\mathsf{baac})
\qquad
\QG=
\Graph{
\GV 1:1,4;
\GV 0:1;
\GV 0.25:2.25,3.5;
\GS 0.25/0:2.25/2,3.5/3,3.5/4;
\GS 0.25/1:2.25/3,3.5/2;
\GS 1/1.5:1/,2/,3/,4/;
\GS 0/-0.5:1/,2/,3/,4/;
\Ptext (0.5,1) {$\mathsf{a}$}
\Ptext (4.5,1) {$\mathsf{d}$}
\Ptext (0.5,0) {$\mathsf{e}$}
\Ptext (1.75,0.35) {$\mathsf{c}$}
\Ptext (4,0.5) {$\mathsf{b}$}
}=\ker(\mathsf{abcd},\mathsf{ecbb})
$$

For partitions on two lines $\PG\in\Part(k,l)$, it will also be convenient to allow having empty blocks. That is, in the diagrams, there may also occur isolated dots.

\subsection{Categories of partitions}
\label{secc.partcat}
We define the following operations on the sets $\Part(k,l)$
\begin{itemize}
\item  The \emph{tensor product} of two partitions $\PG\in\Part(k,l)$ and $\QG\in\Part(k',l')$ is the partition $\PG\otimes \QG\in \Part(k+k',l+l')$ obtained by writing the graphical representations of $\PG$ and~$\QG$ ``side by side''.
$$
\Graph{
\GV 0.5:2.5;
\GV 0:1,4;
\GS 0.5/1:2.5/1,2.5/2,2.5/3;
\GS 0.5/0:2.5/2,2.5/3;
\GS 1/1.5:1/,2/,3/;
\GS 0/-0.5:1/,2/,3/,4/;
}
\otimes
\Graph{
\GV 1:1,4;
\GV 0:1;
\GV 0.25:2.25,3.5;
\GS 0.25/0:2.25/2,3.5/3,3.5/4;
\GS 0.25/1:2.25/3,3.5/2;
\GS 1/1.5:1/,2/,3/,4/;
\GS 0/-0.5:1/,2/,3/,4/;
}
=
\Graph{
\GV 1:5,8;
\GV 0.5:2.5;
\GV 0:1,4,5;
\GV 0.25:6.25,7.5;
\GS 0.5/1:2.5/1,2.5/2,2.5/3;
\GS 0.5/0:2.5/2,2.5/3;
\GS 0.25/0:6.25/6,7.5/7,7.5/8;
\GS 0.25/1:6.25/7,7.5/6;
\GS 1/1.5:1/,2/,3/,5/,6/,7/,8/;
\GS 0/-0.5:1/,2/,3/,4/,5/,6/,7/,8/;
}
$$

\item For $\PG\in\Part(k,l)$, $\QG\in\Part(l,m)$ we define their \emph{composition} $\QG\PG\in\Part(k,m)$ by putting the graphical representation of~$\QG$ below~$\PG$ identifying the lower row of~$\PG$ with the upper row of~$\QG$. The upper row of~$\PG$ now represents the upper row of the composition and the lower row of~$\QG$ represents the lower row of the composition.
$$
\Graph{
\GV 1:1,4;
\GV 0:1;
\GV 0.25:2.25,3.5;
\GS 0.25/0:2.25/2,3.5/3,3.5/4;
\GS 0.25/1:2.25/3,3.5/2;
\GS 1/1.5:1/,2/,3/,4/;
\GS 0/-0.5:1/,2/,3/,4/;
}
\cdot
\Graph{
\GV 0.5:2.5;
\GV 0:1,4;
\GS 0.5/1:2.5/1,2.5/2,2.5/3;
\GS 0.5/0:2.5/2,2.5/3;
\GS 1/1.5:1/,2/,3/;
\GS 0/-0.5:1/,2/,3/,4/;
}
=
\Graph{
\GV 1.5:2.5;
\GV 1:1,4;
\GS 1.5/2:2.5/1,2.5/2,2.5/3;
\GS 1.5/1:2.5/2,2.5/3;
\GS 2/2.5:1/,2/,3/;
\GS 1/0.5:1/,2/,3/,4/;
\GV 0:1,4;
\GV -1:1;
\GV -0.75:2.25,3.5;
\GS -0.75/-1:2.25/2,3.5/3,3.5/4;
\GS -0.75/0:2.25/3,3.5/2;
\GS 0/0.5:1/,2/,3/,4/;
\GS -1/-1.5:1/,2/,3/,4/;
}
= 
\Graph{
\GV 0.5:2.5,3.5,4;
\GV 0:1;
\GS 0.5/1:2.5/1,2.5/2,2.5/3;
\GS 0.5/0:2.5/2,2.5/3,2.5/4;
\GS 1/1.5:1/,2/,3/;
\GS 0/-0.5:1/,2/,3/,4/;
}
$$

\item For $\PG\in\Part(k,l)$ we define its \emph{involution} $\PG^*\in\Part(l,k)$ by reversing its graphical representation with respect to the horizontal axis.
$$
\left(
\Graph{
\GV 0.5:2.5;
\GV 0:1,4;
\GS 0.5/1:2.5/1,2.5/2,2.5/3;
\GS 0.5/0:2.5/2,2.5/3;
\GS 1/1.5:1/,2/,3/;
\GS 0/-0.5:1/,2/,3/,4/;
}
\right)^*
=
\Graph{
\GV 0.5:2.5;
\GV 1:1,4;
\GS 0.5/0:2.5/1,2.5/2,2.5/3;
\GS 0.5/1:2.5/2,2.5/3;
\GS 0/-0.5:1/,2/,3/;
\GS 1/1.5:1/,2/,3/,4/;
}
$$
\end{itemize}

Those operations define the structure of a~rigid monoidal involutive category on the collection $\Part(k,l)$. The set of natural numbers forms the set of objects and the partitions $\PG\in\Part(k,l)$ are the morphisms $k\to l$.

Any collection of subsets $\Cat(k,l)\subset\Part(k,l)$ containing the \emph{identity partition} $\Gid\in\Cat(1,1)$ and the \emph{pair partition} $\Gpair\in\Cat(0,2)$ (playing the role of the duality morphism) that is closed under the category operations and under adding and removing empty blocks is again a~category called a~\emph{category of partitions}.

This definition comes from \cite{BS09}. We made a~slight modification here by allowing partitions to have empty blocks (in \cite{BS09}, the empty blocks that arise in composition are simply deleted). But since we require the categories to be closed under adding and removing empty blocks, the definitions are equivalent.

\subsection{Group-theoretical categories of partitions}
\label{secc.grpthcat}

A category of partitions~$\Cat$ is called \emph{group-theoretical} if $\Ggrpth\in\Cat$. This definition was introduced in~\cite{RW14}. The name comes from a~certain correspondence formulated below.

Let $V$ be a~countable set. We denote by~$\Z_2^{*V}$ the free product~$\Z_2^{*|V|}$, where the generators of the factors are identified with the elements of~$V$. For a~given word $\aw\in V^*$, we denote by $g_{\aw}\in\Z_2^{*V}$ the corresponding group element.

A normal subgroup $A\trianglelefteq\Z_2^{*V}$ is called \emph{$S_V$-invariant} if it is invariant with respect to finitary permutations $V\to V$ (i.e. bijections $V\to V$ that act as identity up to a~finite amount of points). It is called $sS_V$-invariant if it is invariant with respect to arbitrary finitary maps $V\to V$, i.e. arbitrary maps $V\to V$ that act as identity up to a~finite amount of points.

Given a~word $\aw\in V^*$, we denote by $\aw^*\in V^*$ its \emph{reflection}, i.e.\ the word read backwards. Note that $g_{\aw^*}=g_{\aw}^{-1}$.

\begin{thm}[{\cite[Theorem~3.7]{RW14}}]
\label{T.partcor}
Let $V$ be some infinite countable set. We have the following one-to-one correspondence:

Let $A\trianglelefteq\Z_2^{*V}$ be an $sS_V$-invariant normal subgroup. Then
$$\Cat:=\{\ker(\aw,\mathbf{b})\mid g_{\aw^*\bw}\in A\}$$
forms a~group-theoretical category of partitions.

Conversely, let $\Cat$ be a~group-theoretical category of partitions. Then
$$A:=\{g_{\aw^*\bw}\mid\ker(\aw,\bw)\in\Cat\}$$
is an~$sS_V$-invariant normal subgroup of~$\Z_2^{*V}$.
\end{thm}

As an example of an $sS_V$-invariant normal subgroup, we mention the infinite family
$$A_s=\llangle (ab)^s\mid a,b\in V\rrangle\subset\Z_2^{*V},$$
where the double angle brackets denote the smallest \emph{normal} subgroup generated by the given generators. For more instances, see \cite[Example~3.9]{RW14}.

\subsection{Skew categories of partitions}
\label{secc.skewcat}

The definition of a~category of partitions was modified in \cite{Maa18} in order to generalize the above described correspondence to all $S_V$-invariant normal subgroups $A\trianglelefteq\Z_2^{*V}$. We define the following operations that are essentially based on the group multiplication in~$\Z_2^{*V}$

Considering $\PG=\ker(\aw,\bw)$ and $\QG=\ker(\cw,\dw)$, we call the partition $\ker({\aw\cw,\bw\dw})$ a \emph{connected tensor product} of $\PG$ and $\QG$. Note that the result does not only depend on the partitions $\PG$ and $\QG$, but also on the labelling of the blocks. For given $\PG$ and~$\QG$, we can always choose the corresponding labellings $\aw$,~$\bw$, $\cw$,~$\dw$ in such a~way that the words $\aw$ and~$\bw$ do not share any letters with $\cw$ and~$\dw$. In that case, we obtain the standard tensor product as defined in Section~\ref{secc.partcat}. In general, the connected tensor product enables to unite some blocks from~$\PG$ with some blocks from~$\QG$ by choosing a~common letter to denote those blocks.

We define the \emph{restricted composition} to be the ordinary composition from Section~\ref{secc.partcat}, but we restrict only to the pairs of partitions of the form $\PG=\ker(\aw,\bw)$, $\QG=\ker(\bw,\cw)$. Note that we can choose the words $\aw$ and~$\cw$ in such a~way that they share only the letters contained in~$\bw$. In that case, the result of such composition can be written as $\QG\PG=\ker(\aw,\cw)$. In general, we can define the \emph{restricted connected composition} to be the operation $(\ker(\aw,\bw),\ker(\bw,\cw))\mapsto\ker(\aw,\cw)$.

A \emph{skew category of partitions} is a~collection of sets $\Cat(k,l)\subset\Part(k,l)$ containing the identity partition~$\Gid$ and the pair partition~$\Gpair$ that is closed under all possible connected tensor products (with all possible choices of the labelling), restricted compositions, involution, and adding/removing empty blocks. It automatically follows that $\Cat$ is closed under restricted connected compositions.

Skew categories of partitions are again rigid monoidal involutive categories, but this time the set of objects is formed by one-line partitions $\bigcup_{k\in\N_0}\Part(k)$. For a~pair of partitions $p\in\Part(k)$, $q\in\Part(l)$, we define the set of morphisms
$$\Cat(p,q):=\{\ker(\aw,\bw)\in\Cat(k,l)\mid p=\ker(\aw),\;q=\ker(\bw)\}\subset\Cat(k,l).$$

Note that any group-theoretical category of partitions is closed under joining blocks \cite[Lemma~2.3]{RW14}. Consequently, it is closed under connected tensor products and hence forms a~skew category of partitions. On the other hand, skew categories of partitions always contain the partition~$\Ggrpth$. Moreover, group-theoretical categories of partitions precisely correspond to those skew categories of partitions that are invariant with respect to joining blocks.

We have the following generalization of Theorem~\ref{T.partcor}.

\begin{thm}[{\cite[Theorem~2.2]{Maa18}}]
Let $V$ be some infinite countable set. We have the following one-to-one correspondence:

Let $A\trianglelefteq\Z_2^{*V}$ be an $S_V$-invariant normal subgroup. Then
$$\Cat:=\{\ker(\aw,\mathbf{b})\mid g_{\aw^*\bw}\in A\}$$
forms a~skew category of partitions.

Conversely, let $\Cat$ be a~skew category of partitions. Then
$$A:=\{g_{\aw^*\bw}\mid\ker(\aw,\bw)\in\Cat\}$$
is an~$S_V$-invariant normal subgroup of~$\Z_2^{*V}$.
\end{thm}

\subsection{Graph categories}
\label{secc.graphcat}

The goal of this work is to generalize the above results to so-called graph categories defined in \cite{MR19}.

In this article, by a~\emph{graph}, we will always mean a~finite undirected graph with no multiple edges, but with possibility of having loops.

A \emph{bilabelled graph}~$\KG$ is a~triple $(K,{\bf a},{\bf b})$, where $K$ is a~graph, ${\bf a}=(a_1,\dots,a_k)$, ${\bf b}=(b_1,\dots,b_l)$ are tuples of vertices of~$K$. We will call~$\mathbf{a}$ the tuple of \emph{input vertices} while $\mathbf{b}$ are \emph{output vertices} (the role of $\mathbf{a}$ and~$\mathbf{b}$ is switched in comparison with \cite{MR19} to be consistent with the notation for partitions). For any $k,l\in\N_0$ we denote by~$\Gat(k,l)$ the set of all bilabelled graphs $\KG=(K,{\bf a},{\bf b})$ with $|{\bf a}|=k$, $|{\bf b}|=l$ (considering the graphs ``up to isomorphism''). The set of all bilabelled graphs is denoted simply by~$\Gat$.

We define a~structure of a~(monoidal involutive) category on the set of all bilabelled graphs by introducing some operations. Consider $\KG=(K,{\bf a},{\bf b})$, $\HG=(H,{\bf c},{\bf d})$. We define

\begin{itemize}
\item \emph{tensor product} $\KG\otimes \HG=(K\sqcup H,{\bf ac},{\bf bd})$,
\item \emph{composition} (only defined if $|{\bf b}|=|{\bf c}|$) ${\bf H\cdot K}=(H\cdot K,{\bf a},{\bf d})$, where $H\cdot K$ is a~graph that is created from $H\sqcup K$ by contracting the vertex $b_i$ with~$c_i$ for every~$i$ (ignoring the resulting edge multiplicities),
\item \emph{involution} $\KG^*=(K,{\bf b},{\bf a})$.
\end{itemize}

We denote by~$N_k$ the edgeless graph with $k$ vertices. In particular, $N_0$ is the \emph{null graph} containing no vertex. We denote $\nullG:=(N_0,\emptyset,\emptyset)\in\Gat(0,0)$.

We denote $\MG^{k,l}:=(M,v^k,v^l)$, where $M$ is a~graph with a~single vertex~$v$.

Any collection of bilabelled graphs~$\Cat$ containing $\MG^{1,1}$ (playing the role of identity), $\MG^{0,2}$ (playing the role of the duality morphism), $\nullG$ (playing the role of the scalar identity) and closed under the above defined operations forms a~rigid monoidal involutive category called a~\emph{graph category}.


We can represent the bilabelled graphs pictorially in a~similar way as partitions. Unlike \cite{MR19}, we will draw them ``top to bottom'' instead of ``right to left'' to be consistent with partitions.

Considering $\KG=(K,\aw,\bw)\in\Gat(k,l)$ we put $k$ points on a~line and $l$ points on another line below. Between those two lines of points, we draw the graph~$K$ and each input vertex~$a_i$ is connected by a~string to the $i$-th point on the top line and each output vertex~$b_j$ is connected by a~string to the $j$-th point on the bottom line. The graph edges are drawn by thick lines, whereas the above mentioned connecting strings are drawn thin. Typical examples of a~bilabelled graphs may look as follows
$$
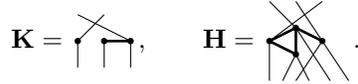
\KG=
\Graph{
\GV 0.5:1,2,3;
\GE 0.5/0.5:2/3;
\GS 0.5/-0.5:1/,2/,3/;
\GS 0.5/1.5:1/2,3/1;
},\qquad
\HG=
\Graph{
\GV 0.5:1,3;
\GV 1:2;
\GV 0:2;
\GE 0/0.5:2/1;
\GE 0.5/1:1/2,3/2;
\GE 0/1:2/;
\GS 1/-1:2/3.3;
\GS 0.5/-1:3/2,3/4,1/1;
\GS 0/-1:2/1.5,2/2.5;
\GS 0.5/2:1/3,3/2;
\GS 1/2:2/1;
}.
$$

The categorical operations then have a~similar pictorial interpretation as in the case of partitions. Tensor product is putting ``side by side'':
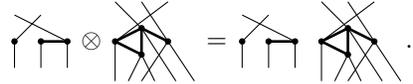
$$
\Graph{
\GV 0.5:1,2,3;
\GE 0.5/0.5:2/3;
\GS 0.5/-0.5:1/,2/,3/;
\GS 0.5/1.5:1/2,3/1;
}\otimes
\Graph{
\GV 0.5:1,3;
\GV 1:2;
\GV 0:2;
\GE 0/0.5:2/1;
\GE 0.5/1:1/2,3/2;
\GE 0/1:2/;
\GS 1/-1:2/3.3;
\GS 0.5/-1:3/2,3/4,1/1;
\GS 0/-1:2/1.5,2/2.5;
\GS 0.5/2:1/3,3/2;
\GS 1/2:2/1;
}
=
\Graph{
\GV 0.5:1,2,3;
\GE 0.5/0.5:2/3;
\GS 0.5/-0.5:1/,2/,3/;
\GS 0.5/1.5:1/2,3/1;
\GV 0.5:4,6;
\GV 1:5;
\GV 0:5;
\GE 0/0.5:5/4;
\GE 0.5/1:4/5,6/5;
\GE 0/1:5/;
\GS 1/-1:5/6.3;
\GS 0.5/-1:6/5,6/7,4/4;
\GS 0/-1:5/4.5,5/5.5;
\GS 0.5/2:4/6,6/5;
\GS 1/2:5/4;
}.
$$
Composition is putting one graph below the other and contracting the strings:
$$
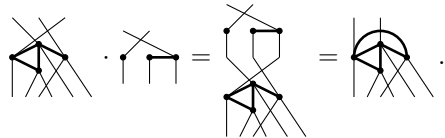
\Graph{
\GV 0.5:1,3;
\GV 1:2;
\GV 0:2;
\GE 0/0.5:2/1;
\GE 0.5/1:1/2,3/2;
\GE 0/1:2/;
\GS 1/-1:2/3.3;
\GS 0.5/-1:3/2,3/4,1/1;
\GS 0/-1:2/1.5,2/2.5;
\GS 0.5/2:1/3,3/2;
\GS 1/2:2/1;
}\cdot
\Graph{
\GV 0.5:1,2,3;
\GE 0.5/0.5:2/3;
\GS 0.5/-0.5:1/,2/,3/;
\GS 0.5/1.5:1/2,3/1;
}
=
\Graph{
\GV 1.5:1,2,3;
\GE 1.5/1.5:2/3;
\GS 1.5/0.5:1/,2/,3/;
\GS 1.5/2.5:1/2,3/1;
\GV -1:1,3;
\GV -0.5:2;
\GV -1.5:2;
\GE -1.5/-1:2/1;
\GE -1/-0.5:1/2,3/2;
\GE -1.5/-0.5:2/;
\GS -0.5/-2.5:2/3.3;
\GS -1/-2.5:3/2,3/4,1/1;
\GS -1.5/-2.5:2/1.5,2/2.5;
\GS -1/0.5:1/3,3/2;
\GS -0.5/0.5:2/1;
}=
\Graph{
\GV 0.5:1,3;
\GV 1:2;
\GV 0:2;
\GE 0/0.5:2/1;
\GE 0.5/1:1/2,3/2;
\GE 0/1:2/;
\GS 1/-1:2/3.3;
\GS 0.5/-1:3/2,3/4,1/1;
\GS 0/-1:2/1.5,2/2.5;
\GEa (1,0.5) (-1,0) (0,1)
\GS 1/2:2/;
\GS 0.5/2:1/;
}.
$$
Involution is vertical reflection.
$$\left(
\Graph{
\GV 0.5:1,2,3;
\GE 0.5/0.5:2/3;
\GS 0.5/-0.5:1/,2/,3/;
\GS 0.5/1.5:1/2,3/1;
}
\right)^*=
\Graph{
\GV 0.5:1,2,3;
\GE 0.5/0.5:2/3;
\GS 0.5/1.5:1/,2/,3/;
\GS 0.5/-0.5:1/2,3/1;
}.$$

In this spirit, the category of partitions embed into the category of all graphs. Given $\PG\in\Part(k,l)$, we associate to it an edgeless bilabelled graph, where the vertices stand for the blocks of~$\PG$. That is $\PG\mapsto (N_b,\aw,\bw)$, where $b$ is the number of blocks in~$\PG$ and $\PG=\ker(\aw,\bw)$. In particular, note that the bilabelled graphs~$\MG^{k,l}$ containing a~single vertex correspond to partitions consisting of a~single block (e.g.~$\MG^{1,1}=\Gid$, $\MG^{0,2}=\Gpair$). For the rest of this article, we will not distinguish between partitions with $k$ upper and $l$ lower points and the associated bilabelled graphs.

%

\begin{rem}
\label{R.vertexrem}
Formally, there is no bijection between categories of partitions and categories of edgeless graphs. The reason is that in case of partitions, we assume by definition that the categories are closed with respect to adding and removing empty blocks. In the language of graphs, empty blocks correspond to isolated vertices. Any graph category is certainly closed under adding isolated vertices as the graph with a single isolated vertex is contained in any category (since $\Gvertex=\Guppair\cdot\Gpair$), but it need not be closed under removing those vertices. In case of graph categories, this can be considered as an unimportant technical detail. However, it becomes important in case of skew categories of partitions and skew graph categories, which we are going to define later.
\end{rem}

\subsection{Rotations and Frobenius reciprocity} Consider a~bilabelled graph $\KG=(K,(a_1,\dots,a_k),(b_1,\dots,b_l))\in\Gat(k,l)$. We define its
\begin{itemize}
\item \emph{left rotation} $\Lrot\KG:=(K,(a_2,\dots,a_k),(a_1,b_1,\dots,b_l))\in\Gat(k-1,l+1)$,
\item \emph{right rotation} $\Rrot\KG:=(K,(a_1,\dots,a_k,b_l),(b_1,\dots,b_{l-1}))\in\Gat(k+1,\ppen l-1)$.
\end{itemize}

It holds that every category of graphs (and hence also every category of partitions) is closed under both those operations and their inverses. The proof is simple: those rotations can actually be realized by composing with some tensor product of the pair partition and identities. For example, we can express
$$\Rrot\KG=(\Gid\otimes\cdots\otimes\Gid\otimes\Guppair)\cdot(\KG\otimes\Gid).$$

As a~consequence, any category of graphs~$\Cat$ is generated by the collection of graphs with output vertices only $\KG\in\Cat(0,k)$, $k\in\N_0$.

\section{Group-theoretical graph categories}
\label{sec.comb}

In this section, we generalize the group-theoretical categories of partitions and skew categories of partitions from \cite{RW14,Maa18} (summarized here in Sects. \ref{secc.grpthcat},~\ref{secc.skewcat}) by introducing the following concepts.

\bigskip
\begin{tabular}{ll}
\textbf{Partition concept}                       & \textbf{Graph analogue}\\\hline
Group-theoretical cat.~of partitions             & Group-theoretical graph category\\
$sS_V$-invariant norm. subgrp. $A\trianglelefteq\Z_2^{*V}$& Easy graph fibration\\
Skew category of partitions                      & Skew graph category\\
$S_V$-invariant norm. subgrp. $A\trianglelefteq\Z_2^{*V}$ & Graph fibration\\
\end{tabular}

\subsection{Graph quotients and overlapping unions}
First, we define some important operations on graphs.

\begin{defn}
Let $K$ be a~graph and $\pi$ a~partition of its vertex set~$V(K)$. We denote by $K/\pi$ the \emph{quotient graph}, where the set of vertices are the blocks of~$\pi$ and there is an edge between two blocks if there is an edge in~$K$ between some of their elements. We denote by~$q_\pi$ the surjection $V(K)\to V(K/\pi)$, which actually defines a~graph homomorphism $q_\pi\colon K\to K/\pi$.
\end{defn}

\begin{defn}
Consider two graphs $K$ and~$H$. An injective partial function $f\colon V(K)\to V(H)$, that is, a~subset $f\subset V(K)\times V(H)$ such that for every $v\in V(K)$ or $v\in V(H)$ there is at most one pair in $f$ containing~$v$ will be called a~\emph{vertex overlap} of $K$ and~$H$. We denote by $K\cup_f H$ the quotient of $K\sqcup H$, where we identify all the pairs in~$f$. We call it the \emph{$f$-union} of $K$ and~$H$.
\end{defn}

If $f$ is empty, then $K\cup_f H=K\sqcup H$ is the disjoint union. Considering non-empty~$f$ corresponds to the situation where the vertex sets $V(K)$ and $V(H)$ overlap and then we compute the true union of $K$ and~$H$. We denote by~$f_K$ the inclusion $V(K)\to V(K\cup_f H)$ and similarly $f_H\colon V(H)\to V(K\cup_f H)$.

\subsection{Graph fibrations}
Before introducing group-theoretical categories of partitions, we first define the structure providing the group theoretical description, which we call the \emph{graph fibration}. Let us start with introducing some notation.

\begin{nota}
Let $V,V'$ be some sets and consider a map $\phi\colon V\to V'$. Then the obvious homomorphisms $V^*\to V'^*$ and $\Z_2^{*V}\to\Z_2^{*V'}$ induced by $\phi$ are denoted by the same letter $\phi$. In particular, the same applies if $\phi$ is a graph homomorphism $G\to G'$ and $V=V(G)$, $V'=V(G')$.
\end{nota}

\begin{defn}
A \emph{graph fibration}~$F$ is a~set of ordered pairs $(K,a)$, where $K$ is a~graph (equivalence class up to isomorphism) and $a\in\Z_2^{*V(K)}$, such that $(N_0,e)$, $(N_1,e)\in F$ ($e$~denotes the group identity) and the following holds.
\begin{enumerate}
\item[(F1)] The sets $F(K):=\{a\in\Z_2^{*V(K)}\mid (K,a)\in F\}$ are either empty or normal subgroups of~$\Z_2^{*V(K)}$.
\item[(F2)] For any automorphism~$\phi$ of~$K$, we have $F(K)=\phi(F(K))$.
\item[(F3)] For any pair of graphs $K$ and~$H$ and any vertex overlap~$f$, $F(K)$ and $F(H)$ embed into $F(K\cup_f H)$ by $f_K$ and~$f_H$. That is, $F$ is closed under the operation of \emph{$f$-union} defined as
$$(K,a)\cup_f (H,b)=(K\cup_f H,f_K(a)f_H(b)).$$
\end{enumerate}
If $F(K)$ is non-empty, we write $K\in F$ and call $F(K)$ a~\emph{fibre} of~$F$. A~graph fibration is called \emph{easy} if it is closed under quotients. That is:
\begin{enumerate}\setcounter{enumi}{3}
\item[(F4)] For any $K\in F$ and any partition~$\pi$ of~$V(K)$, we have $q_\pi(F(K))\subset F(K/\pi)$.
\end{enumerate}
\end{defn}

\begin{rem}
	Let us clarify what do we exactly mean by ``up to isomorphism'' here. The elements of $F$ are formally equivalence classes of the pairs $(K,a)$ with respect to the following equivalence: $(K,a)\equiv(K',a')$ if there is a graph isomorphism $\phi\colon K\to K'$ such that $a'=\phi(a)$.

Nevertheless, in the following text we will not strictly distinguish between the elements and the equivalence classes. In particular, given a graph $K$, we denote by $F(K)$ the set of all $a\in\Z_2^{*V(K)}$ such that the equivalence class of $(K,a)$ is an element of $F$ (exactly as we wrote in (F1)). So, $F(K)$ should always be seen as a normal subgroup of $\Z_2^{*V(K)}$ containing the actual elements of $\Z_2^{*V(K)}$, not any kind of equivalence classes. This then of course already implies that $F(K)$ is invariant with respect to the automorphisms of $K$, so the axiom (F2) is in some sense redundant. Nevertheless, it is convenient to have it explicitly listed since, if we are practically working with the actual graph and not their equivalence classes, we then have to explicitly check that every $F(K)$ is indeed $\Aut K$-invariant.
\end{rem}

\begin{rem}
In case of easy graph fibrations, the axiom~(F3) can be equivalently formulated just for disjoint unions since any $f$-union is a~quotient of the disjoint union.
\end{rem}

\begin{nota}
In the following text, we will usually suppress the maps $f_K$ and~$f_H$. That is, for a~pair of graphs $K$ and~$H$ and a~vertex overlap~$f$, we identify $K$ and~$H$ with the corresponding subgraphs of $K\cup_f H$. If it is clear, which vertex overlap~$f$ are we using, this should cause no confusion.
\end{nota}

\subsection{Group-theoretical graph categories}

In this section, we define what is a \emph{group-theoretical graph category}. The motivation for such a~definition and name is (similarly as in case of partitions) that we have a~certain group-theoretical description of such categories -- namely the easy graph fibrations defined in the previous section. The structure of group-theoretical graph categories is then generalized by defining \emph{skew graph categories}, which then correspond to general graph fibrations.

\begin{defn}
A \emph{group-theoretical graph category} is a~graph category $\Cat\subset\Gat$, which is closed under taking graph quotients.
\end{defn}

\begin{rem}
\label{R.grpthpart}
In terms of partitions, taking quotients means simply joining different blocks. The ability to join blocks is also the key feature of group-theoretical categories of partitions. That is, a~category of partitions~$\Cat$ contains~$\Ggrpth$ (so $\Cat$~is group-theoretical) if and only if it is closed under joining blocks \cite[Lemma~2.3]{RW14}.

That is, there is a correspondence between group-theoretical categories of partitions and group-theoretical graph categories containing edgeless graphs only. In contrast with Remark~\ref{R.vertexrem}, this correspondence is bijective. Indeed, since group-theoretical graph categories are supposed to be closed under quotients, they must in particular be closed under removing isolated vertices.

There is no such a~simple characterization for group-theoretical graph categories as we do for categories of partitions. For instance, it does not hold that any graph category containing the partition $\Ggrpth$ would be group-theoretical. Nevertheless, in Section~\ref{secc.gens}, we are going to formulate a similar characterization in terms of generators.
\end{rem}

In the following, we are going to define skew categories of graphs as a~generalization of skew categories of partitions.

Considering partitions $p\in\Part(k)$, $q\in\Part(l)$, denote
$$\Gat(p,q):=\{\KG=(K,\aw,\bw)\in\Gat(k,l)\mid p=\ker\aw,q=\ker\bw\}\subset\Gat(k,l).$$

We can restrict the composition of bilabelled graphs $\KG=(K,\aw,\bw)$, $\HG=(H,\cw,\dw)$ only to the cases, when $\ker\bw=\ker\cw$ and then interpret the collection of sets $\Gat(p,q)$ as a~monoidal $*$-category with one-line partitions $\bigcup_{k\in\N_0}\Part(k)$ as the set of objects.

\begin{defn}
A \emph{skew graph category}~$\Cat$ is a~set $\Cat\subset\Gat$ of bilabelled graphs containing $\nullG\in\Cat(0,0)$, $\Gid\in\Cat(1,1)$ and $\Gpair\in\Cat(0,2)$, which is closed under
\begin{itemize}
\item \emph{$f$-unions} -- for every $\KG=(K,\aw,\bw)\in\Cat(k,l)$ and $\HG=(H,\cw,\dw)\in\Cat(k',l')$ and for every vertex overlap $f$, we have $\KG\cup_f\HG:=(K\cup_f H,\aw\cw,\bw\dw)\in\Cat(k+k',l+l')$ (for disjoint union, we get the tensor product),
\item \emph{restricted composition} -- for every $\KG=(K,\aw,\bw)\in\Cat(k,l)$ and $\HG=(H,\cw,\dw)\in\Cat(l,m)$ such that $\ker\bw=\ker\cw$, we have $\HG\KG\in\Cat(k,m)$,
\item \emph{involution} -- $\KG=(K,\aw,\bw)\in\Cat(k,l)$ implies $\KG^*=(K,\bw,\aw)\in\Cat(l,k)$.
\end{itemize}
In addition, consider $\KG=(K,\aw,\bw)\in\Cat(k,l)$, $\HG=(H,\cw,\dw)\in\Cat(l,m)$ and a~vertex overlap $f$ of $K$ and~$H$ such that $(b_i,c_i)\in f$ for all~$i$ (existence of which implies $\ker\bw=\ker\cw$). Then we define the \emph{$f$-composition} $\HG\cdot_f\KG:=(K\cup_f H,\aw,\dw)$.
\end{defn}

\begin{rem}
\label{R.skew}
Skew graph categories are indeed rigid monoidal $*$-categories. In particular, the following holds.
\begin{enumerate}
\renewcommand{\theenumi}{\alph{enumi}}
\item For every partition $p\in\Part(k)$, we have an identity morphism $\id_p\in\Cat(p,p)\subset\Cat(k,k)$ that can be created as an appropriate $f$-union of the identities $\Gid\in\Cat(1,1)$.
\item Skew graph categories are closed under left and right rotations. (Proof the same as in \cite[Lemma~1.8(i)]{Maa18}.)
\item Skew graph categories are closed under $f$-compositions. (Proof the same as in \cite[Lemma~1.8(ii)]{Maa18}.)
\end{enumerate}
\end{rem}

\begin{rem}\leavevmode
\label{R.skew2}
\begin{enumerate}
\renewcommand{\theenumi}{\alph{enumi}}
\item Any group-theoretical graph category is also a~skew graph category.
\item Any skew graph category contains~$\Ggrpth$ as a~union of $\Gpair$, $\Gid$, and~$\Guppair$.
\item Nonetheless, not every skew graph category is a~(group-theoretical) graph category as it may not be closed under all compositions.
\end{enumerate}
\end{rem}

Finally, we connect skew graph categories with their group-theoretical description.

\begin{thm}
\label{T.comb}
There is the following one-to-one correspondence between graph fibrations and skew graph categories. For any graph fibration~$F$, we define a~graph category
\begin{equation}
\label{eq.CatF}
\Cat:=\{(K,\aw,\bw)\mid(K,g_\aw^{-1}g_\bw)\in F\}=\{(K,\aw,\bw)\mid (K,g_\bw g_\aw^{-1})\in F\}.
\end{equation}
Conversely, we associate to any skew graph category~$\Cat$ a~graph fibration
\begin{equation}
\label{eq.FCat}
\begin{aligned}
F:={}&\{(K,g_\aw^{-1}g_\bw)\mid (K,\aw,\bw)\in\Cat\}=\{(K,g_\bw g_\aw^{-1})\mid (K,\aw,\bw)\in\Cat\}\\
  ={}&\{(K,g_\aw)\mid (K,\emptyset,\aw)\in\Cat\}.
\end{aligned}
\end{equation}
The graph fibration~$F$ is easy if and only if the skew graph category~$\Cat$ is a~group-theoretical graph category.
\end{thm}
\begin{proof}
Take a~graph fibration~$F$ and construct~$\Cat$ as in \eqref{eq.CatF}. The equality in \eqref{eq.CatF} follows from $F(K)$ being normal. Now, we prove that $\Cat$ is a~skew graph category, that is, closed under the skew category operations.

Unions: Take $\KG=(K,{\bf a},{\bf b})$ and $\HG=(H,{\bf c},{\bf d})$ from $\Cat$ and a~vertex overlap~$f$. We have $(K,g_\aw^{-1}g_\bw)\cup_f(H,g_\dw g_\cw^{-1})=(K\cup_f H,\ppen g_\aw^{-1}g_\bw g_\dw g_\cw^{-1})\in F$ and, from normality of $F(K\cup_f H)$, also $(K\cup_f H,\ppen g_\cw^{-1}g_\aw^{-1}g_\bw g_\dw)=(K\cup_f H,\ppen g_{\aw\cw}^{-1}g_{\bw\dw})\in F$, so $(K\cup_f H,\ppen \aw\cw,\bw\dw)\in\Cat$.

Involution follows from $F(K)$ being closed under the group inversion, so if $(K,g_\aw^{-1}g_\bw)\in F$ then also $(K,g_\bw^{-1}g_\aw)\in F$.

Composition: Take $\KG=(K,\aw,\bw)$, $\HG=(H,\cw,\dw)$, so $(K,g_\aw^{-1}g_\bw)\in F$ and $(H,g_\cw^{-1}g_\dw)\in F$. Suppose first that $\ker\bw=\ker\cw$. Then we can choose a~vertex overlap $f=\{(b_i,c_i)\}_i$, so $f_K(\bw)=f_H(\cw)$. Then $\HG\KG=(K\cup_f H,\aw,\dw)$, which is indeed in~$\Cat$ since we have $(K,g_\aw^{-1}g_\bw)\cup_f(H,g_\cw^{-1}g_\dw)=(K\cup_f H,g_\aw^{-1}g_\dw)$.

Now, let us assume that $F$ is easy, i.e.\ invariant with respect to quotients and let us prove that $\Cat$ is closed under all compositions. So, take arbitrary $\KG=(K,\aw,\bw)\in\Cat(k,l)$, $\HG=(H,\cw,\dw)\in\Cat(l,m)$. Then we have $(K,g_\aw^{-1}g_\bw)\sqcup(H,g_\cw^{-1}g_\dw)=(K\sqcup H,\ppen g_\aw^{-1}g_\bw g_\cw^{-1}g_\dw)\in F$. Now, we can take the quotient of $K\sqcup H$ by identifying the vertices $b_i$ with $c_i$ for every $i$. This then exactly corresponds to the composition $\HG\KG$. Finally, since $F$ is closed under quotients, then $\Cat$ is also closed under quotients.

The converse direction is very straightforward. Consider a~skew graph category~$\Cat$ and prove that \eqref{eq.FCat} defines a~graph fibration. First of all, note that the equalities in \eqref{eq.FCat} follow from skew graph categories being closed under rotations. Now, consider some $K\in F$ and prove that $F(K)$ is a normal subgroup of $\Z_2^{*V(K)}$. First, we prove that $F(K)$ is closed with respect to multiplication. Take $\aw,\bw$ such that $(K,\emptyset,\aw)$, $(K,\emptyset,\bw)\in\Cat$, so $g_\aw,g_\bw\in F(K)$. Then we must have $(K,\emptyset,\aw)\cup_{\rm id}(K,\emptyset,\bw)=(K,\emptyset,\aw\bw)$, so $g_\aw g_\bw\in F(K)$. Secondly, we prove that $F(K)$ is closed on taking inverses. Considering $\KG=(K,\emptyset,\aw)\in\Cat$, we have $\KG^*=(K,\aw,\emptyset)\in\Cat$, so $g_\aw^{-1}\in F(K)$. Finally, the normality. Taking again $\KG=(K,\emptyset,\aw)\in\Cat$ and $v\in V(K)$, then we can construct $\KG\cup_f\Gid=(K,v,\aw v)$, where $f$ is the vertex overlap identifying the vertex of $\Gid$ with $v$. This graph then corresponds to $g_v^{-1}g_\aw g_v$.

Proving that $F$ is closed under $f$-unions and, assuming that $\Cat$ is a~group-theoretical graph category, proving that $F$ is closed under taking quotients, is similarly straightforward.

Finally, we prove that those assignments $\Cat\mapsto F$ and $F\mapsto\Cat$ are inverse to each other, which means that the correspondence is indeed one-to-one. Denote by~$\Phi$ the map $(K,\aw,\bw)\mapsto(K,g_\aw^{-1}g_\bw)$. The correspondence can then be formulated as an assignment $\Cat\mapsto F:=\Phi(\Cat)$ and $F\mapsto\Cat:=\Phi^{-1}(F)$. Note that $\Phi$ is a~well defined mapping, but it is not injective. So, it remains to prove that $\Phi^{-1}(\Phi(\Cat))=\Cat$. In other words, we have to prove that, for every skew category of graphs, $\KG:=(K,\aw,\bw)\in\Cat$ implies $\KG':=(K,\cw,\dw)\in\Cat$ whenever $g_\aw^{-1}g_\bw=g_\cw^{-1}g_\dw$. The equality means that $\KG'$ can differ from $\KG$ by rotation~-- but this is fine as $\Cat$ is closed under rotations -- and by adding or removing some pair $vv$, where $v\in V(K)$, to the input/output sequence. Adding of such a pair can be achieved by computing an $f$-union with~$\Gpair$ (or $\Guppair$ or~$\Gid$), removing can be achieved by \emph{contraction} i.e.\ composing with $\Gid\otimes\cdots\Gid\otimes\Guppair\otimes\Gid\otimes\cdots\otimes\Gid$ (some additional rotations may be needed).
\end{proof}

\begin{rem}[On loops in graphs]
\label{R.loops}
Often, people are interested in simple graphs, that is, graphs without loops. So, there might be a question: Is it essential to consider graphs with loops in the definition of a graph category and graph fibration or can one do the same also without loops?

The easiest way is to actually consider graphs with a loop at every vertex. Everything works perfectly fine if we implicitly assume that every graph appearing in this article has a loop at every vertex. Some statements can even be formulated in a simpler way under this assumption. We will comment on this at the particular places later on in this article. Note however, that this assumption also means that we are slightly modifying the definition of graph categories and graph fibrations since now $N_1$ -- the graph, which is by assumption contained in every graph fibration and every graph category -- denotes the graph with a single vertex \emph{and a loop}.

Now, what if we want to restrict to truly simple graphs only, i.e.\ graphs that have literally no loops. Simple bilabelled graphs are closed under all the operations of skew graph categories. So, in case of skew graph categories, we can again just restrict to simple graphs. However, simple graphs are not closed under taking quotients. So, in case we want to work with ordinary graph categories or, equivalently, easy graph fibrations, then we actually have to construct a quotient category rather than a subcategory. If a~graph operation on simple graphs yields a~non-simple graph, then just declare the result to be equal to zero.

Note that in the easy case, the distinction between the loops-everywhere and no-loops-at-all approaches may not be entirely cosmetic. As we just mentioned, the category operations will act differently. Also the notion of a graph homomorphism differs in those two situations. On the other hand, in the non-easy situation, we deal only with injective homomorphisms, where this distinction disappears.
\end{rem}

\subsection{Full fibrations}

\begin{defn}
A graph fibration~$F$ is called \emph{full} if $K\in F$ for all~$K$.
\end{defn}

\begin{lem}
\label{L.closei}
Let $F$ be a~graph fibration. Suppose $K,H\in F$ are graphs such that there is an embedding (an injective homomorphism) $\iota\colon H\to K$. Then $\iota(F(H))\subset F(K)$.
\end{lem}
\begin{proof}
We have $K=K\cup_\iota H$, so the lemma follows from axiom~(F3).
\end{proof}

\begin{prop}[Alternative definition of full graph fibrations]
\label{P.altfull}
A full graph fibration~$F$ is equivalently a~collection of normal subgroups $F(K)\trianglelefteq\Z_2^{*V(K)}$ for all graphs~$K$, which is closed under injective homomorphisms. That is, for every injective homomorphism of two graphs $\iota\colon H\to K$, we have $\iota(F(H))\subset F(K)$.
\end{prop}
\begin{proof}
	The left-right implication follows from Lemma \ref{L.closei}. For the right-left implication, we get axiom~(F2) by choosing $\phi$ to be the graph automorphism. To check (F3), take any pair of fibres $K,H\in F$ and their vertex overlap $f$. Since $F$ contains all graphs as fibres, it must contain also $K\cup_f H$. Choosing $\iota:=f_K$ in our assumption, we directly have that $F(K)$ embed into $F(K\cup_fH)$ by $f_K$ and the same holds for $H$.
\end{proof}

\begin{prop}[Alternative definition of easy full graph fibrations]
\label{P.alteasy}
An easy full graph fibration~$F$ is equivalently a~collection of normal subgroups $F(K)\trianglelefteq\Z_2^{*V(K)}$ for all graphs~$K$, which is closed under all homomorphisms. That is, for every homomorphism of two graphs $\phi\colon H\to K$, we have $\phi(F(H))\subset F(K)$.
\end{prop}
\begin{proof}
For the left right implication, we just decompose the homomorphism as $\phi=\iota\circ q$, where $q$ is a~surjective homomorphism (that is, a~quotient) and $\iota$~is injective (that is, an embedding). Now we just use axiom~(F4) to deal with~$q$ and Lemma~\ref{L.closei} to deal with~$\iota$.

Most of the right-left multiplication was proven in \ref{P.altfull}, we just have to handle the easiness, i.e.\ axiom (F4). This follows by taking $\phi=q_\pi$.
\end{proof}

%

\begin{prop}
\label{P.easyK}
Let $B$ be a~set of graphs containing $N_0$ and~$N_1$ and closed under quotients and arbitrary $f$-unions. If $B$ contains some graph with at least one edge, then $B$ contains all graphs with loops at every vertex.
\end{prop}
So, if we are computing with the implicit assumption that every graph has a loop at every vertex (see Remark~\ref{R.loops}), then a graph fibration is always full unless it contains edgeless graphs only.
\begin{proof}
If there is any $H\in B$ containing at least one edge, then we can construct a~graph with two vertices connected by an edge as its quotient. Possibly missing loops at those two vertices can be added using some $f$-union with itself and then performing a quotient. From the full graph on two vertices (with loops everywhere), any graph~$K$ (with loops everywhere) can be constructed using $f$-unions.
\end{proof}

\subsection{Examples}
Let $V:=\{\la_1,\la_2,\dots\}$ be an infinite countable set.

\begin{ex}[Group-theoretical categories of partitions]
\label{E.easypart}
Consider $A$ an $sS_V$-invariant normal subgroup of~$\Z_2^{*V}$. Let $\Cat_A$ be the associated category of partitions. Denote by $V_n:=\{\la_1,\dots,\la_n\}$ the first $n$~letters in~$V$ and by $A_n:=A\cap\Z_2^{*V_n}$ the corresponding subgroup of~$A$ using only those letters $\la_1,\dots,\la_n$. Then we can define an easy graph fibration 
$$F_A:=\{(N_n,g)\mid g\in A_n,\;n\in\N_0\},$$
where $N_n$ is the edgeless graph with $n$ vertices identified with the letters $\la_1,\dots,\la_n$. The associated graph category can then be identified with the category of partitions~$\Cat_A$.

Conversely, as we already mentioned in Remark~\ref{R.grpthpart}, a group-theoretical graph category $\Cat$ where all graphs have no edges can always be identified with a group-theoretical category of partitions. Such a category then corresponds to some graph fibration $F$ with fibres $A_k:=F(N_k)$. Axiom (F3) in this case means that the groups $A_k$ are closed under injective homomorphisms $A_k\to A_l$, $k\le l$ mapping generators to generators, so we have a chain of embeddings $\cdots A_k\subset A_{k+1}\subset\cdots$. Axiom (F4) means that the groups are closed also under the surjective homomorphisms $A_l\to A_k$, $k\le l$ mapping generators to generators, so, in particular, each group is just a subgroup of the union $A:=\bigcup_kA_k$ generated by $\{\la_1,\dots,\la_k\}$.
\end{ex}

\begin{ex}[Skew categories of partitions]
Let $A$ be an $S_V$-invariant normal subgroup of~$\Z_2^{*V}$. Then the above mentioned construction defines a~(possibly non-easy) graph fibration~$F_A$. The associated skew graph category then coincides with the skew category of partitions corresponding to~$A$.
\end{ex}

\begin{ex}[``New'' skew categories of partitions]
\label{E.new}
In this case, the converse does not hold. A~skew graph category $\Cat$ containing edgeless graphs only may not be closed under removing isolated vertices, in which case we cannot identify it with a skew category of partitions in the sense of \cite{Maa18}. Let us look what happens on the level of the group-theoretical description. Let $F$ be the corresponding graph fibration with fibres $A_k:=F(N_k)$. Again, axiom (F3) means that we have a chain of embeddings $\cdots A_k\subset A_{k+1}\subset\cdots$. But now we are missing axiom (F4) which would allow us to project from right to left. So, $F$ may not be determined by a single group $A\trianglelefteq\Z_2^{*V}$.

Let us bring a concrete example suggested to the author by Laura Maaßen: Consider the skew graph category $\Cat$ generated by $\GLabc$. That is, $A_3$ is the normal subgroup of $\Z_2^{*3}$ generated by $\la_1\la_2\la_3$ and its permutations; moreover, $A_k$ for any $k\ge 3$ is the normal subgroup of $\Z_2^{*k}$ generated by $\la_i\la_j\la_k$ with $i,j,k\in\{1,\dots,k\}$ mutually distinct. For $k<3$ is $A_k$ trivial. We can again try to take $A:=\bigcup_k A_k$, but this no more determines all the sets $A_k$. Indeed, we have $\la_1\la_2=(\la_1\la_3\la_4)\cdot(\la_4\la_3\la_2)\in A_4\subset A$, but obviously $\la_1\la_2\not\in A_2$! (And one can actually also prove that $\la_1\la_2\not\in A_3$.)

We can look on this example in terms of pictures. Starting with the generator $\GLabc$, we can do the connected tensor product with itself to obtain
$\Graph{
\GV 0.5:1,3.5,6;
\GV 1:3.5;
\GS 0.5/0:1/,6/;
\GS 0.5/0.2:3.5/3,3.5/4;
\GS 1/0.4:3.5/2,3.5/5;
\GS 0.2/0:3/,4/;
\GS 0.4/0:2/,5/;}$.
Now, a composition with
$\Graph{
\GV 0.5:1,3.5,6;
\GV 0:3.5;
\GS 0/1:1/,6/;
\GS 0.5/0.8:3.5/3,3.5/4;
\GS 0/0.6:3.5/2,3.5/5;
\GS 0.8/1:3/,4/;
\GS 0.6/1:2/,5/;
}$
gives us
$\Graph{
\GV 0.5:1,2;
\GV 1:1.2,1.8;
\GS 0.5/0:1/,2/;
}$.
So, in the end, 
$\Graph{
\GV 0.5:1,2;
\GV 1:1.2,1.8;
\GS 0.5/0:1/,2/;
}\in\Cat$,
but
$\Graph{
\GV 0.5:1,2;
\GS 0.5/0:1/,2/;
}\not\in\Cat.$

Modifying the definition of a skew graph category, namely by removing the assumption that the category is closed under removing empty blocks, we can ``repair'' the correspondence. Using this modified definition, it will again hold that skew categories of partitions exactly correspond to the skew graph categories containing edgeless graphs only. It depends on the application whether it is more convenient to use the former simpler definition, where a skew category is determined by a single group $A\trianglelefteq\Z_2^{*V}$, or whether one needs the more general definition with a sequence of groups. Note that the latter may be necessary to describe certain quantum group semidirect products $\hat\Gamma\rtimes S_n$, see Remark~\ref{R.converse}.
\end{ex}

\begin{ex}[Homogeneous graph fibration]
Let $A$ be an $S_V$-invariant normal subgroup of~$\Z_2^{*V}$ and let $B$ be a~set of graphs closed under arbitrary $f$-unions. For every graph $K\in B$, consider some injection $\iota_K\colon V(K)\to V$. Denote $A_K:=\iota_K^{-1}(A)$. Since $A$ is $S_V$-invariant, it does not matter, how precisely we define the injection~$\iota_K$. Then we can define a~graph fibration $F_{A,B}$ by
$$F_{A,B}(K)=\begin{cases}A_K&\text{if }K\in B,\\\emptyset&\text{otherwise.}\end{cases}$$
The associated graph category can be described as follows
$$\Cat_{A,B}=\{\KG=(K,\aw,\bw)\mid K\in B,\;\ker(\aw,\bw)\in\Cat_A\}.$$
\end{ex}

\begin{ex}[Easy full homogeneous graph fibration]
	Let $A$ be $sS_V$-invariant normal subgroup of~$\Z_2^{*V}$. Then we can construct an easy full graph fibration $F_{A,{\rm full}}:=F_{A,B}$ taking $B$ to be the set of all graphs. So, $F_{A,{\rm full}}(K)=A_K$ for every graph~$K$.
\end{ex}

In the following example, we construct an easy full graph fibration that is not determined by a~single normal subgroup $A\trianglelefteq\Z_2^{*V}$.

\begin{ex}[There is more!]
\label{E.abab}
For every graph~$K$, we define $F(K)$ to be the normal subgroup of~$\Z_2^{*V(K)}$ generated by words of the form $abab$, where $a,b$ are any adjacent vertices in~$K$, so
$$F(K)=\llangle abab\mid \{a,b\}\in E(K)\rrangle\trianglelefteq\Z_2^{*V(K)}.$$
Here, by the double brackets, we denote the normal subgroup generated by the given generators. Now since adjacency is preserved under $f$-unions and quotients, we have that $F$ is indeed closed under those operations and hence it is an easy graph fibration. 
\end{ex}

\subsection{Generators of categories}
\label{secc.gens}
In this section, we study group-theoretical graph categories and skew graph categories in terms of generators. We answer the following two questions. First, given a~set of generators, find an explicit description of the group-theoretical category or the skew category it generates. Secondly, we characterize group-theoretical categories among all graph categories in terms of their generators.

\begin{defn}
Let $S$ be a~set of bilabelled graphs. We denote by $\langle S\rangle$ resp.\ $\langle S\rangle\grpth$ resp.\ $\langle S\rangle^{\text{skew}}$ the graph category resp.\ group-theoretical graph category resp.\ skew graph category \emph{generated} by~$S$. That is, the smallest category containing~$S$.
\end{defn}

\begin{defn}
Let $S$ be a~set of graphs. We define its \emph{closure under arbitrary $f$-unions} to be the smallest set of graphs containing~$S$ and closed under $f$-unions.
\end{defn}

In this article, we will typically consider generating sets $S$ that contain the one-vertex graph $N_1$ as this is required by the definition of a graph fibration. (We should formally also always include the empty graph $N_0$, but this has of course virtually no impact on the closure under $f$-unions.)

\begin{ex}
Take the graph of two adjacent vertices $K_2=\Gedge$. Then the closure of $\{N_1,K_2\}$ is simply the set of all graphs (as was already used in the proof of Proposition~\ref{P.easyK}).

Take now the triangle $K_3=\Gtriangle$. The closure of $\{N_1,K_3\}$ is the set of all graphs, where every edge is incident to some triangle. (Indeed, take such a graph $G$ and denote by $n$ the number of its vertices. We can reconstruct $G$ as follows. By repeated unions of $N_1$, construct the edgeless graph $N_n$. Now, go through every edge in $G$ and use the $f$-union to add the whole triangle this edge is incident with to our graph we are constructing (it does not matter which one if there is more). This way, we add all the necessary edges. For the converse direction, it is easy to see that the described set is indeed closed under $f$-unions.)
\end{ex}

In the following propositions, the double angle brackets again denote the normal subgroup generated by the given generators.

\begin{prop}
\label{P.skewgen}
Let $S$ be a~set of bilabelled graphs. Denote by~$\Cat$ the skew graph category generated by~$S$ and by~$F$ the corresponding graph fibration. Then $F$ contains as fibres the closure of $N_0$, $N_1$, and the underlying graphs in~$S$ under arbitrary $f$-unions. For every $K\in F$, we then have
$$F(K)=\llangle \iota(g_{\aw^*\bw})\mid (H,\aw,\bw)\in S;\; \iota\colon H\to K\text{ an embedding}\rrangle\trianglelefteq\Z_2^{*V(K)}$$
\end{prop}
\begin{proof}
The fact that $F$ contains as fibres (at least) the $f$-union closure of graphs in~$S$ follows from axiom (F3) of graph fibrations being closed under $f$-unions. Now the $\supset$ inclusion in the equality follows from Lemma~\ref{L.closei}.

To prove the inclusion~$\subset$, it is enough to show that the right-hand side defines a~graph fibration. By definition, $F(K)$ is a normal subgroup so (F1) holds. If we take an embedding $\iota\colon H\to K$, then $\phi\circ\iota$ is surely also an embedding for any $\phi\in\Aut K$, so (F2) holds. Finally, taking two fibres $K,K'\in F$ and their vertex overlap $f$, then $K\cup_fK'$ is again a fibre by definition. Any element of $F(K)$ is of the form $a=\iota(g_{\aw^*\bw})$. Since $f_K\circ\iota$ is surely an injective homomorphism, we have that $f_K(a)=(f_K\circ\iota)(g_{\aw^*\bw})\in F(K\cup_fK')$. Similarly for $f_{K'}$.
\end{proof}

The situation in the easy case, i.e.\ with group-theoretical graph categories, gets a~bit simpler if we are focusing only on graphs with a loop at every vertex (see Remark~\ref{R.loops}). Recall from Proposition~\ref{P.easyK} that in this case a~group-theoretical graph category either coincides with some group-theoretical category of partitions or it corresponds to a~full graph fibration.

\begin{prop}
\label{P.grpthgen}
Let $S$ be a~set of bilabelled graphs containing at least one graph with at least one edge. Denote by~$\Cat$ the group-theoretical graph category generated by~$S$ and by~$F$ the corresponding full easy graph fibration. Then
$$F(K)=\llangle \phi(g_{\aw^*\bw})\mid (H,\aw,\bw)\in S;\; \phi\colon H\to K\text{ a~homomorphism}\rrangle\trianglelefteq\Z_2^{*V(K)}$$
for every graph~$K$ with loops at every vertex.
\end{prop}
\begin{proof}
Follows directly from Propositions~\ref{P.alteasy}, \ref{P.easyK}.
\end{proof}

Now, we look on the question of generators of group-theoretical graph categories. First we need a~small lemma.

\begin{lem}
\label{L.grpthpart}
Let $\Cat$ be a~graph category containing~$\Ggrpth$. Then $\Cat$ is closed under the following operations.
\begin{enumerate}
\item Moving a~pair of neighbouring output strings coming from a~common vertex:
$$(K,\emptyset,\aw xx\bw\cw)\leftrightarrow (K,\emptyset,\aw\bw xx\cw).$$
\item Computing quotients $\KG\mapsto\KG/\pi$, where all vertices of~$\KG$ that are not input/output are singletons in~$\pi$ (i.e.\ they are not identified with any other vertex).
\end{enumerate}
\end{lem}
\begin{proof}
Moving the pair of strings one position left resp.\ right is achieved by composing with
$\Graph{
\GV 0.5:-2,0,2,4,6;
\GV 0.8:1.4;
\GS 0.5/1:2/2,2/3;
\GS 0.5/0:2/1,2/2;
\GS 0.8/1:1.4/1;
\GS 0.2/0:2.6/3;
\GSa (1.4,0.8) (-0.6,0.3) (0.25,0.25)
\GS 0/1:-2/,0/,4/,6/;
\Ptext (-0.9,0.5) {\dots}
\Ptext (5.1,0.5) {\dots}
}$
resp.
$\Graph{
\GV 0.5:-2,0,2,4,6;
\GV 0.8:2.6;
\GS 0.5/1:2/2,2/1;
\GS 0.5/0:2/3,2/2;
\GS 0.8/1:2.6/3;
\GS 0.2/0:1.4/1;
\GSa (2.6,0.8) (0.6,0.3) (-0.25,0.25)
\GS 0/1:-2/,0/,4/,6/;
\Ptext (-0.9,0.5) {\dots}
\Ptext (5.1,0.5) {\dots}
}$.

Identifying two distinct output vertices is achieved by composing with partition
$\Graph{
\GV 0.5:1,3,5,6,7,9,11;
\GS 0/1:1/,3/,5/,7/,9/,11/;
\GS 0.8/1:4/,8/;
\GS 0.2/0:4/,8/;
\GS 0.2/0.8:4/8,8/4;
\Ptext (2.1,0.5) {\dots}
\Ptext (10.1,0.5) {\dots}
\Ptext (6.1,0.8) {\dots}
\Ptext (6.1,0.2) {\dots}
}$, which is generated by~$\Ggrpth$ (see \cite[page~6]{RW14}).
\end{proof}

\begin{prop}
\label{P.grpthgen2}
Let $S$ be a~set of bilabelled graphs such that every vertex appears at least once within the input/output vertices. Then the graph category $\langle \Ggrpth,S\rangle$ is group-theoretical. Consequently, $\langle\Ggrpth,S\rangle=\langle S\rangle\grpth$.
\end{prop}
\begin{proof}
Given any bilabelled graph $\KG=(K,\aw,\bw)$, we associate to it a~bilabelled graph $\tilde\KG:=(K,\aw,\bw\vw)$, where $\vw=\lv_1\lv_1\lv_2\lv_2\cdots\lv_m\lv_m$ and $\{\lv_1,\dots,\lv_m\}$ denotes the set of all vertices in~$K$. It holds that for any $\KG\in S$, we have $\tilde\KG\in\langle\Ggrpth,S\rangle$. Indeed, recall that for every $v\in V(K)$, we assume that it is an input/output vertex. We can create two extra input/output strings coming from the vertex by composing with~$\Gforkk=\Ggrpth\cdot(\Gpair\otimes\Gid)$. Then, we can move those to the end of the output tuple by Lemma~\ref{L.grpthpart}(1).

Now, we prove that actually for any $\KG\in\langle\Ggrpth,S\rangle$, we have $\tilde\KG\in\langle\Ggrpth,S\rangle$. It is enough to show that for every $\KG,\HG\in\langle\Ggrpth,S\rangle$, such that $\tilde\KG,\tilde\HG\in\langle\Ggrpth,S\rangle$, we also have $\widetilde{\KG\otimes\HG},\widetilde{\HG\cdot\KG},\widetilde{\KG^*}\in\langle\Ggrpth,S\rangle$. This is straightforward to check using Lemma~\ref{L.grpthpart}(1).

Finally, it follows from Lemma~\ref{L.grpthpart}(2) that $\langle\Ggrpth,S\rangle$ is closed under arbitrary graph quotients. Indeed, taking any $\KG\in\langle\Ggrpth,S\rangle$, we can construct~$\tilde\KG$, then use Lemma~\ref{L.grpthpart}(2) to construct the corresponding quotient and finally delete the redundant output strings by composing with~$\Guppair$.
\end{proof}

We can also reformulate the above proposition as an equivalence.

\begin{prop}
Let $\Cat$ be a graph category. It is group-theoretical if and only if there exists a~set of bilabelled graphs $S$ whose vertices all appear in the input/output tuples such that $\Cat=\langle\Ggrpth,S\rangle$.
\end{prop}
\begin{proof}
The right-left implication follows directly from Proposition~\ref{P.grpthgen2}. To prove the left-right implication, take any generating set $\tilde S$ of $\Cat$. Computing $f$-unions of graphs in $\tilde S$ with $\Gpair$, we are able to make all vertices of the graphs appear in the input/output tuples while not changing the category they generate. Finally, we already mentioned in Remark~\ref{R.skew2} that every group-theoretical graph-category must contain $\Ggrpth$ as an $f$-union of $\Guppair$, $\Gid$ and~$\Gpair$.
\end{proof}

\begin{ex}
\label{E.abab2}
As an example, consider the category $\langle\Ggrpth,\Gabab\rangle$. From Proposition~\ref{P.grpthgen2}, it follows that it is a~group-theoretical graph category -- the smallest one containing the bilabelled graph~$\Gabab$. Now using Proposition~\ref{P.grpthgen}, we find out that this category is described by the graph fibration from Example~\ref{E.abab}.
\end{ex}

\section{Preliminaries: Quantum groups and Tannaka--Krein duality}
\label{sec.prelimQG}

In this section, we introduce very briefly the theory of compact matrix quantum groups, Tannaka--Krein duality, and the connection with diagram categories. For more detailed introduction, see for example \cite{Web17,Fre19survey}.

\subsection{Compact matrix quantum groups}
\label{secc.qgdef}

An \emph{orthogonal compact matrix quantum group} is a pair $G=(A,u)$, where $A$ is a $*$-algebra and $u\in M_n(A)$ a~matrix such that
\begin{enumerate}
\item the elements $u_{ij}$ $i,j=1,\dots, n$ generate~$A$,
\item the matrix $u$ is orthogonal, i.e.\ $u_{ij}=u_{ij}^*$ and $uu^t=1_n=u^tu$,
\item the map $\Delta\colon A\to A\otimes_{\rm min} A$ defined as $\Delta(u_{ij}):=\sum_{k=1}^n u_{ik}\otimes u_{kj}$ extends to a~$*$-homomorphism.
\end{enumerate}

This formulation comes from \cite{Fre19survey}. We could consider more general \emph{compact matrix quantum groups} by weakening axiom (2) (instead of orthogonality assuming only that $u$ and $u^t$ are unitarizable). This concept was first introduced by Woronowicz in \cite{Wor87}. Compact matrix quantum groups generalize the notion of compact matrix groups in the following sense.

\begin{ex}[Compact matrix group $H$]
Let $H$ be an (orthogonal) compact matrix group. Consider the coordinate functions $v_{ij}\colon H\to\C$ defined by $v_{ij}(h)=h_{ij}$ for $h\in H$. Let $A:=O(H)$ be the coordinate algebra, i.e.\ the $*$-algebra generated by $v_{ij}$. Then $(A,v)$ forms an (orthogonal) compact matrix quantum group. Conversely, any compact matrix quantum group $H=(A,v)$, where $A$ is a~commutative $*$-algebra, is of this form.
\end{ex}

For this reason, given any compact matrix quantum group $\GGr=(A,u)$, we denote $O(\GGr):=A$ (interpreting the elements as non-commutative functions). The matrix~$u$ is called the \emph{fundamental representation} of~$\GGr$. However, there is also a~\emph{dual} viewpoint coming from the following example.

\begin{ex}[Finitely generated group $\Gamma$]
\label{E.Gamma}
Let $\Gamma$ be a~finitely generated discrete group. Denote by $g_1,\dots,g_n$ its generators. Let $\C\Gamma$ be the associated group $*$\hbox{-}algebra, denote by $\gamma_1,\dots,\gamma_n\in\C\Gamma$ its elements corresponding to the generators of~$\Gamma$ and by~$\gamma$ the diagonal matrix $\gamma=\diag(\gamma_1,\dots,\gamma_n)$. Then $\hat\Gamma=(\C\Gamma,\gamma)$ is also a~compact matrix quantum group. If $g_i^2=e$, i.e.\ $\gamma_i^2=1$, then it is an orthogonal CMQG. It is called the \emph{compact dual} of~$\Gamma$.
\end{ex}


We say that a~compact matrix quantum group $\HGr=(O(\HGr),v)$ is a~\emph{quantum subgroup} of $\GGr=(O(\GGr),u)$, denoted $\HGr\subset\GGr$, if there is a~surjective $*$-homomorphism $O(\GGr)\to O(\HGr)$ mapping $u_{ij}\mapsto v_{ij}$. That is, quantum subgroup can be constructed by \emph{adding relations to the algebra}.

In this article, we focus on compact matrix quantum groups of the following form.

\begin{ex}[Semidirect product $\hat\Gamma\rtimes H$]
Let $H\subset S_n$ be a~permutation group represented by permutation matrices. That is, $H$~can be considered as a~compact matrix quantum group $H=(O(H),v)$, where $v_{ij}\colon H\to\C$ is the function defined by $v_{ij}(\sigma)=\delta_{i\sigma(j)}$. Note in particular that $v_{ij}^2=v_{ij}$ and $\sum_k v_{ik}=1=\sum_k v_{kj}$ for every $i,j$. Let $\Gamma$ be a~finitely generated group with generators $g_1,\dots,g_n$ and let us define the matrix $\gamma=\diag(\gamma_1,\dots,\gamma_n)\in M_n(\C\Gamma)$, where $\gamma_i$ correspond to~$g_i$, so $\hat\Gamma=(\C\Gamma,\gamma)$ is a~compact matrix quantum group as discussed above.

It holds that $\GGr:=\hat\Gamma H:=(\C\Gamma\otimes O(H),\gamma v)$ is a~compact matrix quantum group. In particular, if we assume that $g_i^2=e$, i.e. $g_i^{-1}=g_i$, so $\gamma_i^2=1$ and $\gamma_i^*=\gamma_i$, then $\GGr$ is an orthogonal compact matrix quantum group.

To see that, denote $u:=\gamma v$, i.e.\ $u_{ij}=\gamma_iv_{ij}$. The orthogonality is obvious since both $\gamma$ and~$v$ are orthogonal matrices. Consequently, the axiom (2) holds. We have that $\gamma_i=\sum_k u_{ik}$ and $v_{ij}=u_{ij}^2$, so the axiom (1) holds. Finally, axiom (3) can be proven by showing that the comultiplication~$\Delta$ corresponds to a~semidirect product construction with respect to the (co)action of $H\subset S_n$ on~$\Gamma$ by permuting the generators. See \cite[Sect.~2.5]{RW15} for details (also see e.g.~\cite[Sect.~10.2.6]{KS97} for the definition of the double crossed product). In order to emphasize this structure, we will denote this quantum group by $\hat\Gamma\rtimes H:=\GGr$.
\end{ex}

\subsection{Representation categories and Tannaka--Krein duality}

In this paper, by a~representation category, we mean a~linear rigid monoidal $*$-category, where the monoid of objects are the natural numbers~$\N_0$ and morphisms are realized as linear operators. To be more concrete:

Consider $n\in\N$. A~\emph{representation category} is a~collection of vector spaces
$$\RCat(k,l)\subset\Lin((\C^n)^{\otimes k},(\C^{n})^{\otimes l}),\quad k,l\in\N_0$$
such that the following holds:
\begin{enumerate}
\item For $T\in\RCat(k,l)$, $T'\in\RCat(k',l')$, we have $T\otimes T'\in\RCat(k+k',l+l')$.
\item For $T\in\RCat(k,l)$, $S\in\RCat(l,m)$, we have $ST\in\RCat(k,m)$.
\item For $T\in\RCat(k,l)$, we have $T^*\in\RCat(l,k)$
\item For every $k\in\N_0$, we have $1_n^{\otimes k}\in\RCat(k,k)$.
\item We have $\sum_{k=1}^n e_k\otimes e_k\in\RCat(0,2)$
\end{enumerate}

Given a~compact matrix quantum group~$\GGr$ with fundamental representation~$u$ of size $n\times n$, we define
$$\RCat_\GGr(k,l):=\Mor(u^{\otimes k},u^{\otimes l}):=\{T\colon (\C^n)^{\otimes k}\to(\C^n)^{\otimes l}\mid Tu^{\otimes k}=u^{\otimes l}T\}.$$
It is easily checked that $\RCat_\GGr$ is a~representation category in the sense of our definition. The so-called Tannaka--Krein duality says also the converse:

\begin{thm}[Woronowicz--Tannaka--Krein \cite{Wor88}]
For every representation category~$\RCat$ there exists a~unique orthogonal compact matrix quantum group~$\GGr$ with $\RCat_\GGr=\RCat$.
\end{thm}

Recall that quantum subgroups $\HGr\subset\GGr$ are defined by the fact that $O(\HGr)$ has more relations than $O(\GGr)$. Consequently, we have that $\HGr\subset\GGr$ if and only if $\RCat_\HGr\supset\RCat_\GGr$.

\subsection{CMQGs associated to diagram categories}

Let us fix a~graph~$G$ and label its vertices by numbers $1,\dots,n:=|V(G)|$. For every bilabelled graph $\KG=(K,\aw,\bw)\in\Gat(k,l)$ we define a~linear map $T^G_\KG\colon(\C^n)^{\otimes k}\to (\C^n)^{\otimes l}$ by the following formula
$$[T^G_{\KG}]_{\jw\iw}:=\#\{\text{homomorphisms }\phi\colon K\to G\mid \phi(\aw)=\iw,\;\phi(\bw)=\jw\}$$
for any pair of multiindices $\iw$, $\jw$ with $i_1,\dots,i_k,j_1,\dots,j_l\in\{1,\dots,n\}\simeq V$.

\begin{thm}[{\cite[Sect.~3.2]{MR19}}]
Let $G$ be a~graph. The assignment $\KG\mapsto T^G_\KG$ is a~monoidal unitary functor. That is,
\begin{enumerate}
\item $T^G_{\KG\otimes\HG}=T^G_\KG\otimes T^G_\HG$,
\item $T^G_{\HG\KG}=T^G_\HG T^G_\KG$,
\item $T^G_{\KG^*}=T^{G*}_\KG$.
\end{enumerate}
\end{thm}

Consequently, for any graph category $\Cat\subset\Gat$, its image under~$T^G$, i.e.\ the collection
$$\RCat(k,l):=\spanlin\{T^G_\KG\mid \KG\in\Cat(k,l)\}\subset\Lin((\C^n)^{\otimes k},(\C^n)^{\otimes l})$$
forms a~representation category. We can apply Woronowicz--Tannaka--Krein theorem to this category and associate a~compact matrix quantum group to it.

\begin{cor}[{\cite[Theorem~8.2]{MR19}}]
Let $G$ be a~graph and $\Cat$ a~graph category. Then there exists a~unique orthogonal compact matrix quantum group $\GGr=(O(\GGr),u)$ such that
$$\Mor(u^{\otimes k},u^{\otimes l})=\spanlin\{T^G_{\KG}\mid\KG\in\Cat(k,l)\}.$$
\end{cor}

We can actually construct the quantum group very explicitly. The associated $*$\hbox{-}algebra can be defined as the universal $*$-algebra satisfying the intertwiner relations $u^{\otimes k}T=Tu^{\otimes l}$:
$$O(\GGr)=\staralg(u_{ij};\;i,j=1,\dots,n\mid T^G_{\KG}u^{\otimes k}=u^{\otimes l}T^G_\KG\;\forall\KG\in\Cat(k,l);\;k,l\in\N_0).$$
Actually, thanks to the Frobenius reciprocity, we may only consider the bilabelled graphs with output vertices only:
$$O(\GGr)=\staralg(u_{ij};\;i,j=1,\dots,n\mid u^{\otimes k}T^G_\KG=T^G_\KG\;\forall\KG\in\Cat(0,k);\;k\in\N_0).$$

Those ideas formulated in \cite{MR19} constitute a~generalization of the work \cite{BS09}, which used categories of partitions to construct quantum groups. Recall that the category of all partitions embed into the category of all graphs. Let $G$ be a~graph and $n:=|V(G)|$. Consider a~partition $\PG\in\Part(k,l)$, which can also be interpreted as an edgeless bilabelled graph $\PG\in\Gat(k,l)$. Then regardless the structure of~$G$, we have $T^G_\PG=T^{(n)}_\PG$, where the entries of $T^{(n)}_\PG$ are given by ``blockwise Kronecker delta'' $[T^{(n)}_\PG]_{\jw\iw}=\delta_{\PG}(\iw,\jw)$. That is, assign the $k$ points in the upper row of~$p$ the numbers $i_1,\dots,i_k$ (from left to right) and the $l$ points in the lower row $j_1,\dots,j_l$ (again from left to right). Then $\delta_\PG(\iw,\jw)=1$ if the points belonging to the same block are assigned the same numbers. Otherwise $\delta_\PG(\iw,\jw)=0$. 

To bring an example, recall the partitions $p$ and~$q$ from Equation~\ref{eq.pq}.
$$
\PG=
\Graph{
\GV 0.5:2.5;
\GV 0:1,4;
\GS 0.5/1:2.5/1,2.5/2,2.5/3;
\GS 0.5/0:2.5/2,2.5/3;
\GS 1/1.5:1/,2/,3/;
\GS 0/-0.5:1/,2/,3/,4/;
}\in\Part(3,4)
\qquad
\QG=
\Graph{
\GV 1:1,4;
\GV 0:1;
\GV 0.25:2.25,3.5;
\GS 0.25/0:2.25/2,3.5/3,3.5/4;
\GS 0.25/1:2.25/3,3.5/2;
\GS 1/1.5:1/,2/,3/,4/;
\GS 0/-0.5:1/,2/,3/,4/;
}\in\Part(4,4)
$$
In this case, we have
$$\delta_{\PG}(\iw,\jw)=\delta_{i_1i_2i_3j_2j_3},\qquad\delta_\QG(\iw,\jw)=\delta_{i_2j_3j_4}\delta_{i_3j_2}.$$

\begin{ex}[Important graph categories and the associated quantum groups]
In this paper, we will use the following two important results: Let $G$ be a~graph with $n$ vertices.

The category of all partitions~$\Part$ corresponds to the group of all permutations~$S_n$ represented by permutation matrices \cite{HR05,BS09}. That is, considering $n\in\N$, denote by~$A_\sigma$ the permutation matrix corresponding to a~permutation $\sigma\in S_n$. Then
$$\{T\colon(\C^n)^{\otimes k}\to(\C^n)^{\otimes l}\mid TA_\sigma^{\otimes k}=A_\sigma^{\otimes l}T\;\forall\sigma\in S_n\}=\spanlin\{T^{(N)}_\PG\mid \PG\in\Part(k,l)\}.$$

The category of all graphs~$\Gat$ corresponds to the group of all automorphisms $\Aut G$ of the given graph~$G$ \cite{MR19}.
That is, intertwiners of $\Aut G$ are given by
$$\{T\colon(\C^n)^{\otimes k}\to(\C^n)^{\otimes l}\mid TA_\sigma^{\otimes k}=A_\sigma^{\otimes l}T\;\forall\sigma\in \Aut G\}=\spanlin\{T^{G}_\KG\mid \KG\in\Gat(k,l)\}.$$
\end{ex}

\subsection{CMQGs associated to group-theoretical categories of partitions}

The quantum groups corresponding to group-theoretical categories of partitions were identified in \cite{RW15}:

\begin{thm}[{\cite[Theorem~4.5]{RW15}}]
\label{T.easypartQG}
Consider a~natural number $n\in\N$. Let $\Cat$ be a~group-theoretical category of partitions. Denote by $A\trianglelefteq\Z_2^{*\infty}$ the associated normal subgroup and by $A_n\trianglelefteq\Z_2^{*n}$ its subgroup using only $n$ generators (as in Example~\ref{E.easypart}). Denote $\Gamma:=\Z_2^{*n}/A_n$. Then the quantum group associated to~$\Cat$ is of the form
$$\GGr=\hat\Gamma\rtimes S_n.$$
\end{thm}

Now a~natural question is, what if $A$ is $S_\infty$-invariant, but not $sS_\infty$-invariant? In that case, the quantum group $\GGr:=\hat\Gamma\rtimes S_n$ does not correspond to any category of partitions. But can we describe the associated intertwiner spaces in a~different way? This question is answered in \cite{Maa18}, where skew categories of partitions were introduced for this purpose.

Let $\PG\in\Part(k,l)$ be a~partition and $n\in\N$. We define
$$\hat\delta_\PG^{(n)}(\iw,\jw)=\begin{cases}1&\PG=\ker(\iw,\jw)\\0&\text{otherwise.}\end{cases}$$
This defines a~linear map $\hat T^{(n)}_\PG\colon(\C^n)^{\otimes k}\to(\C^n)^{\otimes l}$ by $[\hat T^{(n)}_\PG]_{\jw\iw}:=\hat\delta_\PG(\iw,\jw)$.

The assignment $\PG\mapsto\hat T^{(n)}_\PG$ is not a~functor! Despite this, the following holds.

\begin{thm}[{\cite[Theorem~3.5]{Maa18}}]
Consider a~collection of subsets $\Cat(k,l)\subset\Part(k,l)$. Then
$$\RCat(k,l):=\spanlin\{\hat T^{(n)}_\PG\mid\PG\in\Cat(k,l)\}$$
forms a~representation category if and only if $\Cat$ is a~skew category of partitions.
\end{thm}

\begin{thm}[{\cite[Theorem~4.17]{Maa18}}]
Consider a~natural number $n\in\N$. Let $\Cat$ be a~skew category of partitions. Denote by $A\trianglelefteq\Z_2^{*\infty}$ the associated normal subgroup and by $A_n\trianglelefteq\Z_2^{*n}$ its subgroup using only $n$~generators (as in Example~\ref{E.easypart}). Denote $\Gamma:=\Z_2^{*n}/A_n$. Then the quantum group associated to the representation category from the previous theorem is of the form
$$\GGr=\hat\Gamma\rtimes S_n.$$
\end{thm}

\section{CMQGs associated to group-theoretical graph categories}
\label{sec.QG}

\subsection{The easy case}
Let $G$ be a graph. In the following theorem, we characterize the compact matrix quantum group associated through the functor $T^G$ to a~group-theoretical graph category $\Cat$ corresponding to some graph fibration $F$ assuming that $G\in F$. Recall from Proposition~\ref{P.easyK} that, if $G$ has a loop at every vertex, then $G\in F$ is satisfied automatically unless $F$ contains edgeless graphs only (which would mean that $\Cat$ is a category of partitions and this case is already handled by Theorem~\ref{T.easypartQG}).

\begin{thm}
\label{T.Geasy}
Let $G$ be a~graph, $F$~an easy graph fibration with $G\in F$. Denote by~$\Cat$ the group-theoretical graph category associated to~$F$ and by~$\GGr$ the quantum group associated to $\Cat$ and~$G$. Then
$$\GGr=\hat\Gamma\rtimes\Aut G,$$
where $\Gamma=\Z_2^{*V(G)}/F(G)$.
\end{thm}
\begin{proof}
To make sense of the statement and the proof, we need to identify the vertices of~$G$ with numbers $1,\dots,n:=|V(G)|$. Let us denote by~$u$ the fundamental representation of~$\GGr$. We have $\Ggrpth\in\Cat$ and hence $\GGr\subset\Z_2^{*n}\rtimes S_n$. Therefore, we can write $u=\gamma v$, where $\gamma=\diag(\gamma_1,\dots,\gamma_n)$ with $\gamma_1,\dots,\gamma_n$ corresponding to the generators of~$\Z_2^{*n}$ and $v_{ij}$~corresponding to the generators of~$S_n$. So, $\gamma_k$ commute with~$v_{ij}$, the $v_{ij}$ are even central projections, and we can express $v_{ij}=v_{ij}^2=u_{ij}^2$ and $\gamma_i=\sum_j \gamma_iv_{ij}=\sum_j u_{ij}$.

Before going into the proof, let us denote by $\xi^\aw:=T^G_{(G,\emptyset,\aw)}$ the intertwiner associated to the bilabelled graph $(G,\emptyset,\aw)$ for some given word $\aw$. Recall that the entries of this vector are given by
$$\xi^\aw_\iw =\#\hbox{homomorphisms }\phi\colon G\to G\hbox{ such that }\phi(\aw)=\iw.$$

First, we are going to prove the inclusion~$\subset$. We need to show that the generators~$\gamma_i$ satisfy the relations of~$\Gamma$ and that the generators~$v_{ij}$ satisfy the relations of $\Aut G$. Let us start with the latter. 
Denote $\GG:=(G,\emptyset,\vw)$ and $\tilde\GG:=(G,\emptyset,\tilde\vw)$, where $\vw=(1,2,\dots,n)$ and $\tilde\vw=(1,1,2,2,\dots,n,n)$. Since $g_{\tilde\vw}=e$ and since $G\in F$, we surely have $\tilde\GG\in\Cat(0,2n)$. We can write
$$\tilde\GG=\Gfork\!\!^{\otimes n}\cdot\GG,$$
so the relation associated to $\tilde\GG\in\Cat(0,2n)$ can be written as
$$(u\otimes u)^{\otimes n}T_{\Gfork}^{\otimes n}\xi^\vw=T_{\Gfork}^{\otimes n}\xi^\vw.$$
One can easily check that $T_{\Gmerge}(u\otimes u)T_{\Gfork}=v$ since $u_{ij}^2=v_{ij}$. So, multiplying the relation above with $T_{\Gmerge}^{\otimes n}$ from left, we get exactly $v^{\otimes n}\xi^\vw=\xi^\vw$. Substituting $v$ by some permutation matrix $A_\sigma$ corresponding to a permutation $\sigma\in S_n$, we need to show that the relation implies that $\sigma\in\Aut G$. In terms of the entries, the relation then reads $\xi^\vw_\iw=[A_\sigma\xi^\vw]_\iw=\xi^\vw_{\sigma^{-1}(\iw)}$. Note that $\xi^\vw_\iw$ equals either one or zero depending on whether $j\mapsto i_j$ defines a homomorphism $G\to G$. Hence, if we choose $\iw:=(1,2,\dots,n)$, the relation exactly says that $\sigma$ should be an automorphism $G\to G$.

So, we have just proven that $v$ represents some subgroup of $\Aut G$. Consequently, its representation category contains all intertwiners associated to any bilabelled graph. In particular, $v^{\otimes k}\xi^{\aw}=\xi^\aw$ for any word $\aw$ with $g_\aw\in F(G)$ (this can actually be proven also directly by modifying the proof above).

Now, we derive the relations for the generators~$\gamma_i$. We need to show that $\gamma_{\aw}=\gamma_{a_1}\cdots \gamma_{a_k}$ equals identity for every word~$\aw$ such that $g_\aw\in F(G)$. So, take some word~$\aw$ such that $g_\aw\in F(G)$, for which we have the relation $\xi^\aw=u^{\otimes k}\xi^\aw=\gamma^{\otimes k}v^{\otimes k}\xi^\aw=\gamma^{\otimes k}\xi^\aw$, so
$$\gamma_{i_1}\cdots\gamma_{i_k}\xi^\aw_\iw =[\gamma^{\otimes k}\xi^\aw]_\iw =\xi^\aw_\iw .$$
If we choose $\iw :=\aw$, we surely have $\xi^\aw_\aw\neq 0$, so we can divide by this number and get $\gamma_{a_1}\cdots\gamma_{a_k}=1$, which is what we wanted.

To prove the opposite inclusion, we need to take arbitrary $\KG\in\Cat(k,l)$ and show that the relation $\tilde u^{\otimes l}T_\KG^G=T_\KG^G\tilde u^{\otimes k}$ is satisfied in $\hat\Gamma\rtimes\Aut G$. Here, we denote by~$\tilde u$ the fundamental representation of $\hat\Gamma\rtimes\Aut G$. Let us denote also $\tilde u=\tilde\gamma\tilde v$, where $\tilde\gamma$~is the fundamental representation of~$\hat\Gamma$ and $\tilde v$~is the fundamental representation of $\Aut G$. From Frobenius reciprocity, it is enough to consider bilabelled graphs~$\KG$ with output vertices only. For any such $\KG=(K,\emptyset,\aw)\in\Cat(0,k)$, denote by~$\xi^\aw_K$ the corresponding intertwiner. Since $v$ is the fundamental representation of $\Aut G$, we have again $\tilde u^{\otimes k}\xi^\aw_K=\tilde\gamma^{\otimes k}\xi^\aw_K$. The relation $[\tilde\gamma^{\otimes k}\xi^\aw_K]_\iw =[\xi^\aw_K]_\iw $ can be written as
$$\tilde\gamma_{i_1}\cdots\tilde\gamma_{i_k}=1\quad\hbox{whenever $\exists\phi\colon K\to G$ such that $\phi(\aw)=\iw $.}$$
So, we need to show that $g_\iw\in F(G)$ for all words $\iw\in V(G)^*$ such that there exists a~graph~$K$ and a~word $\aw\in V(K)^*$ such that $g_\aw\in F(K)$ and a~graph homomorphism $\phi\colon K\to G$ such that $\phi(\aw)=\iw$. But this directly follows from Proposition \ref{P.alteasy}.
\end{proof}

\begin{ex}
\label{E.abab3}
Recall the category $\langle\Ggrpth,\Gabab\rangle$ from Examples \ref{E.abab},~\ref{E.abab2}. Considering a~graph~$G$, this category corresponds to the semidirect product $\hat\Gamma\rtimes\Aut G$, where $\Gamma$ is a~group generated by some~$g_v$, $v\in V(G)$ such that $g_v^2=e$ for every~$v$ and $g_vg_w=g_wg_v$ if there is an edge $\{v,w\}\in E(G)$.

For example, consider $G=
\Graph{
\GV 1:1;
\GV 0:1;
\GV 0.5:2;
\GE 0/1:1/;
}$. In this case, $\Aut G=S_2\times E$, $\Gamma=(\Z_2\times\Z_2)*\Z_2$, so
$$\GGr=\hat\Gamma\rtimes\Aut G=\widehat{((\Z_2\times\Z_2)*\Z_2)}\rtimes(S_2\times E)=H_2\mathbin{\hat*}\Z_2,$$
where $H_2=\Z_2\wr S_2=(\Z_2\times\Z_2)\rtimes S_2$ is the \emph{hyperoctahedral group} and $\mathbin{\hat*}$~denotes the quantum group (dual) free product as defined by Wang~\cite{Wan95free}.
\end{ex}

\subsection{The $\hat T$ maps}
\label{secc.That}

In the following subsections, we are going to interpret the skew categories of graphs in terms of quantum groups. We start with generalizing the idea of the maps~$\hat T$.

%

\begin{defn}
Let $G$ be a~graph. Denote $n:=|V(G)|$ and label all vertices of~$G$ by numbers $1,\dots,n$. For every bilabelled graph $\KG=(K,\aw,\bw)\in\Gat(k,l)$, we define the map $\hat T_\KG^G\colon (\C^n)^{\otimes k}\to(\C^n)^{\otimes l}$ by
$$[\hat T_\KG^G]_{\jw\iw}:=\#\{\hbox{injective homomorphisms }\phi\colon K\to G\mid\phi(\aw)=\iw,\phi(\bw)=\jw\}.$$
\end{defn}

\begin{rem}
\label{R.Thatzero}
Considering some bilabelled graph $\KG=(K,\aw,\bw)$, we have $\hat T_\KG=0$ unless $K$~is a~subgraph of~$G$. In particular, we have $\hat T_\KG=0$ whenever $V(K)>V(G)$.
\end{rem}

The following result constitutes a~generalization of \cite[Lemma~3.2]{Maa18}.

\begin{prop}
\label{P.That}
Consider $\KG=(K,\aw,\bw)\in\Gat(k,l)$, $\HG=(H,\aw',\bw')\in\Gat(k',l')$. Then the following holds.
\begin{enumerate}
\item $\hat T_{\KG}^G\otimes\hat T_{\HG}^G=\sum_f \hat T_{\KG\cup_f\HG}^G$.
\item $\hat T_{\HG}^G\hat T_{\KG}^G=\sum_f\hat T_{\HG\cdot_f\KG}^G$ (supposing $k'=l$). In particular $\hat T_{\HG}^G\hat T_{\KG}^G=0$ if $\ker\bw\neq\ker\aw'$.
\item $(\hat T_{\KG}^G)^*=\hat T_{\KG^*}^G$.
\end{enumerate}
\end{prop}
\begin{proof}
To prove item (1), we need to show for all multiindices $\iw,\jw,\iw',\jw$ that
\begin{align*}
&\#\left\{(\phi,\psi)\mathrel\Big|\begin{matrix}\phi\colon K\to G\hbox{ inj.\ hom.; }\phi(\aw)=\iw,\phi(\bw)=\jw\\\psi\colon H\to G\hbox{ inj.\ hom.; }\psi(\aw')=\iw',\psi(\bw')=\jw'\end{matrix}\right\}\\&=\sum_f\#\{\omega\colon K\cup_f H\to G\hbox{ inj.\ hom.}\mid\omega(\aw\aw')=\iw\iw',\omega(\bw\bw')=\jw\jw'\}.
\end{align*}
For every $(\phi,\psi)$, we can define a~vertex overlap $f:=\{(v,w)\mid\phi(v)=\psi(w)\}\subset V(K)\times V(H)$ and an injective homomorphism $\omega\colon K\cup_f H\to G$ by gluing $\phi$ and~$\psi$ together. Conversely, taking a~vertex overlap~$f$ and an injective homomorphism $\omega\colon K\cup_f H\to G$, we can define $\phi\colon K\to G$ and $\psi\colon H\to G$ by $\phi(v)=\omega(v)$ and $\psi(w)=\omega(w)$. We see that the assignment $(\phi,\psi)\leftrightarrow (\omega,f)$ is a~bijection. From this the equality follows.

We use the same approach to prove (2).
\begin{align*}
[\hat T_\HG^G\hat T_\KG^G]_{\jw\iw}&=\sum_\mathbf{k}\#\left\{(\phi,\psi)\Bigm|\begin{matrix}\phi\colon K\to G\hbox{ inj.\ hom.; }\phi(\aw)=\iw,\phi(\bw)=\mathbf{k}\\\psi\colon H\to G\hbox{ inj.\ hom.; }\psi(\aw')=\mathbf{k},\psi(\bw')=\jw\end{matrix}\right\}\\
&=\sum_\mathbf{k}\sum_f\#\left\{\omega\colon K\cup_f H\to G\hbox{ inj.\ hom.}\Bigm|\begin{matrix}\omega(\aw)=\iw,\omega(\bw')=\jw,\\\omega(\bw=\aw')=\mathbf{k}\end{matrix}\right\}\\
&=\sum_f\#\{\omega\colon K\cup_f H\to G\hbox{ inj.\ hom.}\mid\omega(\aw)=\iw,\omega(\dw)=\jw\}=\sum_f[T_{\HG\cdot_f\KG}^G]_{\jw\iw}.
\end{align*}

Item (3) can be seen directly from the definition of~$\hat T_\KG^G$.
\end{proof}

Let us also formulate the connection between the intertwiners $T^G_\KG$ and~$\hat T^G_\KG$ by generalizing \cite[Lemma~4.21]{Maa18}.

\begin{prop}
\label{P.Moebius}
Let $G$ be a~graph, $\KG$~a~bilabelled graph. Then
$$T^G_\KG=\sum_{\pi\in\Part(V(K))}\hat T^G_{\KG/\pi}.$$
\end{prop}
\begin{proof}
Recall that
$$[T^G_\KG]_{\jw\iw}=\#\{\text{homomorphisms }\phi\colon K\to G\text{ such that $\phi(\iw)=\aw$, $\phi(\jw)=\bw$}\}.$$
Now every such homomorphism~$\phi$ can be decomposed as $\phi=\iota\circ q_\pi$, where blocks of~$\pi$ are the sets of vertices that are mapped onto a~single vertex, so $\pi=\{\phi^{-1}(v)\mid v\in V(G)\}$, and $\iota\colon K/\pi\to G$ is an injective homomorphism. This proves the inequality $[T^G_\KG]_{\jw\iw}\le\sum_\pi[\hat T^G_{\KG/\pi}]_{\jw\iw}$.

For the opposite inequality, note that considering any partition $\pi\in\Part(V(K))$ and any injective homomorphism $\iota\colon K/\pi\to G$, we can construct a~homomorphism $\iota\circ q_\pi=:\phi\colon K\to G$. Moreover, distinct pairs $(\pi,\iota)$ define distinct homomorphisms~$\phi$.
\end{proof}

\subsection{Representation category associated to a skew graph category}
The maps~$\hat T$ allow us to associate a~representation category to any skew graph category.

\begin{defn}
Let $G$ be a~graph and $\Cat$~a~skew graph category. We denote
$$\Cat^G(k,l):=\spanlin\{\hat T_\KG^G\mid \KG\in\Cat(k,l)\}\subset\Lin((\C^n)^{\otimes k},(\C^n)^{\otimes l}),$$
where $n=|V(G)|$.
\end{defn}

\begin{prop}
$\Cat^G$ is a~representation category.
\end{prop}
\begin{proof}
Follows from Proposition~\ref{P.That}.
\end{proof}

An important point is that this generalizes the easy case, which uses the maps~$T$ instead of~$\hat T$.
\begin{prop}
\label{P.CatGeq}
Let $G$ be a~graph and $\Cat$ a~group-theoretical category of graphs. Then
$$\Cat^G=\spanlin\{\hat T^G_\KG\mid\KG\in\Cat(k,l)\}=\spanlin\{T^G_\KG\mid\KG\in\Cat(k,l)\}.$$
\end{prop}
\begin{proof}
Recall that a~group-theoretical category of graphs is by definition closed under graph quotients. Therefore, the inclusion $\supset$ directly follows from Proposition~\ref{P.Moebius}. The matrix of coefficients of $(T^G_\KG)_{\KG\in\Cat(k,l)}$ expressed as linear combinations of $(\hat T^G_\KG)_{\KG\in\Cat(k,l)}$ is triangular with non-zero diagonal. So, the transformation can be inverted, which provides a proof for the opposite inclusion $\subset$.
\end{proof}

In the following, we are going to find a~linear basis for the morphism spaces $\Cat^G(k,l)$ (compare with \cite[Lemma~3.4]{Maa18}).

\begin{defn}
\label{D.WG}
Let $G$ be a~graph. We denote by $W_G(k):=V(G)^k/\Aut G$ the set of equivalence classes of words~$\aw$ over~$V(G)$ of length $k\in\N_0$ up to automorphisms of~$G$. We also denote $W_G(k,l):=V(G)^k\times V(G)^l/\Aut G$. For given $\aw,\bw$, we denote by $[\aw]\in W_G(k)$ and $[\aw,\bw]\in W_G(k,l)$ the corresponding equivalence classes. For a~skew category~$\Cat$, we denote $W_G^\Cat(k,l):=\{[\aw,\bw]\in W_G(k,l)\mid (G,\aw,\bw)\in\Cat(k,l)\}$.
\end{defn}

\begin{defn}
\label{D.xiG}
Let $G$ be a~graph. We denote $\xi_G^\aw:=\hat T_{(G,\emptyset,\aw)}^G$ for every $\aw\in V(G)^*$. That is,
$$[\xi_G^\aw]_\iw=\#\{\hbox{automorphisms }\phi\colon G\to G\mid\phi(\aw)=\iw\}.$$
Similarly, we define $\hat T_G^{\aw\bw}:=\hat T_{(G,\aw,\bw)}^G$ for every $\aw,\bw\in V(G)^*$.
\end{defn}

We obviously have $\xi_G^\aw=\xi_G^{\aw'}$ if $\aw'=\phi(\aw)$ for some automorphism $\phi\colon G\to G$. Similar equality holds for $\hat T_G^{\aw\bw}$. So, these tensors can be assigned to the whole equivalence classes $[\aw]\in W_G(k)$, resp.\ $[\aw,\bw]\in W_G(k,l)$.

\begin{lem}
\label{L.ThatLI}
Let $G$ be a~graph. The set $\{\hat T^{\aw\bw}_G\mid [\aw,\bw]\in W_G(k,l)\}$ is linearly independent for every $k,l\in\N_0$.
\end{lem}
\begin{proof}
Follows from the fact that $[\hat T^{\aw\bw}_G]_{\jw\iw}[\hat T^{\aw'\bw'}_G]_{\jw\iw}=0$ unless $[\aw,\bw]=[\aw',\bw']$ in $W_G(k,l)$.
\end{proof}

\begin{prop}
\label{P.ThatLB}
Let $G$ be a~graph, $\Cat$ a~skew graph category such that $(G,\emptyset,\emptyset)\in\Cat$. The vector space $\Cat^G(k,l)$ has a~linear basis $\{\hat T^{\aw\bw}_G\mid [\aw,\bw]\in W_G^\Cat(k,l)\}$.
\end{prop}
\begin{proof}
The linear independence was proven in Lemma \ref{L.ThatLI}. Obviously, we have the inclusion. It remains to prove that the set generates indeed the whole space. To prove that, take arbitrary $\HG\in\Cat(k,l)$. Using Proposition \ref{P.That}, we have
$$\alpha\hat T_\HG=\hat T_\HG\otimes\hat T_{(G,\emptyset,\emptyset)}=\sum_f\hat T^G_{\HG\cup_f(G,\emptyset,\emptyset)}=\sum_{\substack{\phi\colon H\to G\\\text{ inj.\ hom.}}}\hat T^{\phi(\aw)\phi(\bw)}_G,$$
where $\alpha=\hat T_{(G,\emptyset,\emptyset)}=\left|\Aut G\right|\neq 0$. For the last equality, recall Remark \ref{R.Thatzero}. We have $\hat T_{\HG\cup_f(G,\emptyset,\emptyset)}=0$ unless $f$ identifies {\em all} the vertices of~$H$ with some vertices of~$G$. That is, $f=\{(v,\phi(v))\mid v\in H\}$, where $\phi$ is some embedding $H\to G$. Then, of course, $H\cup_f G=G$.
\end{proof}

Now, we make the result more general getting rid of the assumption $(G,\emptyset,\emptyset)\in\Cat$ (which is equivalent to $G\in F$ for the associated graph fibration~$F$).

\begin{lem}
\label{L.greatestK}
Let $F$ be a~set of graphs closed under arbitrary $f$-unions. Then, for every graph~$G$, there exists a~greatest subgraph~$K$ of~$G$ contained in~$F$. That is, there is $K\in F$ such that $K\subset G$ and, for every $G\supset H\in F$, we have $H\subset K$.

Moreover, the embedding $K\to G$ is determined uniquely up to automorphisms of~$K$.
\end{lem}
\begin{proof}
Suppose there were two different maximal graphs $K_1,K_2\in F$ contained in~$G$ and denote by $\iota_1,\iota_2$ the corresponding embeddings. Then their union inside~$G$, that is, $K:=\iota_1(K_1)\cup\iota_2(K_2)$ would also satisfy the assumptions, but it would be larger than both $K_1$ and~$K_2$, which is a~contradiction.

The second part is proven the same way. If there were two embeddings $\iota_1,\iota_2\colon K\to G$ such that $\iota_1(K)\neq\iota_2(K)$, then the union of their images would define a~greater subgraph of~$G$ contained in~$F$, which is a~contradiction.
\end{proof}

\begin{rem}
\label{R.Kvert}
If $F$ contains the single vertex graph $N_1$, then the number of vertices in $K$ coincides with the number of vertices in $G$. Indeed, if $K$ had fewer vertices, then $K\sqcup N_1$ would still be a subgraph of $G$ contained in $F$.
\end{rem}

\begin{prop}
\label{P.ThatLB2}
Let $G$ be a~graph, $\Cat$ a~skew graph category corresponding to a~graph fibration~$F$. Denote by~$K$ the greatest subgraph of~$G$ contained in~$F$. Then the vector space $\Cat^G(k,l)$ has a~linear basis $\{\hat T^G_{(K,\aw,\bw)}\mid [\aw,\bw]\in W_K^\Cat(k,l)\}$.
\end{prop}
\begin{proof}
Again, we have that $T^G_{(K,\aw,\bw)}$ is well defined as it does not depend on the particular representative of $[\aw,\bw]$. The proof of linear independence is the same as in Lemma~\ref{L.ThatLI}. To prove that the set generates $\Cat^G(k,l)$, take arbitrary $\HG=(H,\aw,\bw)\in\Cat(k,l)$. Then
$$\alpha\hat T^G_\HG=\hat T^G_\HG\otimes\hat T_{(K,\emptyset,\emptyset)}=\sum_f\hat T^G_{(H\cup_f K,\aw,\bw)}=\sum_{\substack{\phi\colon H\to K\\\text{ inj.\ hom.}}}\hat T^G_{(K,\aw,\bw)},$$
where $\alpha\neq 0$ is the number of embeddings $K\to G$. Let us again explain the last equality. In the sum on the left hand side, we have either $H\cup_fK\subset G$, so, from maximality of~$K$, we must have $H\cup_fK=K$, or we have $H\cup_fK\not\subset G$ and hence $T^G_{(H\cup_f K,\aw,\bw)}=0$.
\end{proof}

\subsection{Quantum group associated to a skew graph category}
Now we are about to formulate the main result of this article interpreting the skew graph categories in terms of quantum groups. We start with a simplified version of the result.

\begin{prop}
\label{P.G}
Let $F$ be a~set of graphs containing $N_0$ and $N_1$ and closed under arbitrary $f$-unions. Setting $F(H)=\Z_2^{*V(H)}$ for all $H\in F$ turns it into a~graph fibration corresponding to a~skew category $\Cat(k,l)=\{(H,\aw,\bw)\mid H\in F,\ppen\aw\in V(K)^k,\ppen\bw\in V(K)^l\}$. Let $G$ be a graph and denote by $K$ the greatest subgraph of $G$ contained in $F$. Then the quantum group corresponding to the category~$\Cat^G$ is the group $\Aut K$.
\end{prop}
\begin{proof}
Let us denote the associated (quantum) group by~$\GGr$. In order to make sense of the statement, we need to fix the inclusion $\iota\colon K\to G$ and identify the vertex sets $V:=V(K)=V(G)$ (recall from Remark~\ref{R.Kvert} that $|V(K)|=|V(G)|$) so that $\Aut K\subset S_V=S_{V(G)}$. First of all, note that $\Cat$ contains all partitions $\Part\subset\Cat$. Consequently, the associated quantum group is a~subgroup of the symmetric group $\GGr\subset S_V$.

First, we prove that $\GGr\supset\Aut K$. So, take an element $\sigma\in\Aut K$ and prove that the associated permutation matrix $A_\sigma$ (with entries $[A_\sigma]_{ij}=\delta_{i\sigma(j)}$) satisfies all the intertwiner relations coming from $\Cat^G$. Thanks to the Frobenius reciprocity, it is enough to consider the intertwiners $\xi\in\Cat^G(0,k)$. According to Proposition~\ref{P.ThatLB2}, $\Cat^G(0,k)$ has a linear basis of elements $\xi^\aw:=\hat T^G_{(K,\emptyset,\aw)}$. So, we need to show that $A_\sigma^{\otimes k}\xi^\aw=\xi^\aw$ for every $\aw\in V(K)^k$ and every $k\in\N_0$. This relation means
$$[\xi^\aw]_\iw=[A_\sigma^{\otimes k}\xi^\aw]_\iw=[\xi^\aw]_{\sigma^{-1}(\iw)},$$
that is,
\begin{align*}
&\#\{\text{inj.\ hom.\ }\phi\colon K\to G\text{ such that }\phi(\aw)=\iw\}\\
&=\#\{\text{inj.\ hom.\ }\phi\colon K\to G\text{ such that }\sigma(\phi(\aw))=\iw\}.
\end{align*}
This is surely satisfied by every $\sigma\in\Aut K$.

For the converse inclusion $\GGr\subset\Aut K$, we need to take any $\sigma\in S_{V(K)}$ satisfying the intertwiner relations of $\Cat^G$ and show that actually $\sigma\in\Aut K$. We do this by a smart choice for the intetwiner: we take $\vw:=(\lv_1,\dots,\lv_n)$, where $\lv_1,\dots,\lv_n$ is the list of all the vertices in $V$.
\begin{align*}
	1=[\xi^\vw]_\vw&=[A_\sigma\xi^\vw]_\vw=[\xi^\vw]_{\sigma^{-1}(\vw)}\cr&=
\begin{cases}
1&\text{if $\sigma^{-1}$ is an injective homomorphism $K\to G$}\cr
0&\text{otherwise}
\end{cases}
\end{align*}
\end{proof}

\begin{thm}
\label{T.G}
Let $F$ be a~graph fibration and $\Cat$~the corresponding category. Let $G$ be a~graph, let $K$ be the greatest subgraph of~$G$ contained in~$F$. Then the quantum group corresponding to~$\Cat^G$ is
$$\GGr=\hat\Gamma\rtimes\Aut K,$$
where $\Gamma=\Z_2^{*V(K)}/F(K)$.
\end{thm}
\begin{proof}
Again, we need to fix the inclusion $\iota\colon K\to G$, so that we can identify the vertices of $K$ with the vertices of $G$, i.e.\ $V:=\{\lv_1,\dots,\lv_n\}:=V(K)=V(G)$. The proof now closely follows the proof of Theorem~\ref{T.Geasy}. Denote by~$\GGr$ the quantum group corresponding to~$\Cat$. Again, we must have $\GGr\subset\Z_2^{*n}\rtimes S_n$ as $\Ggrpth\in\Cat$. So, denote $u=\gamma v$ the fundamental representation of~$\GGr$.

To prove the inclusion $\GGr\subset\hat\Gamma\rtimes\Aut K$, we have to show that the generators~$\gamma_i$ satisfy the relations of~$\Gamma$ and that $v_{ij}$ satisfy the relations of $\Aut K$. For the first part, the proof is completely the same as in case of Theorem~\ref{T.Geasy}.

For the second part, put $\tilde\vw=(\lv_1\lv_1\lv_2\lv_2\cdots\lv_m\lv_m)$. We surely have $e=g_{\tilde\vw}\in F(K)$, so $\tilde\KG:=(K,\emptyset,\tilde\vw)\in\Cat(0,2m)$. Denote also $\KG:=(K,\emptyset,\vw)$ with $\vw=(\lv_1\lv_2\cdots \lv_m)$. We can write
$$\tilde\KG=\Gfork\!\!^{\otimes m}\cdot\KG,$$
so the relation associated to $\tilde\KG\in\Cat(0,2m)$ can be written as
$$(u\otimes u)^{\otimes m}T_{\Gfork}^{\otimes m}\xi^\vw=T_{\Gfork}^{\otimes m}\xi^\vw,$$
where we denote $\xi^\vw=\hat T^G_{(K,\emptyset,\vw)}$ and $T_{\Gfork}=T^G_{\Gfork}=\hat T^G_{\Gfork}$.

Recall that $v_{ij}=v_{ij}^2=u_{ij}^2$ and hence $v=T_{\Gmerge}(u\otimes u)T_{\Gfork}$. Thus, multiplying the above relation by $T_{\Gmerge}^{\otimes m}$ from the right, we obtain
$$v^{\otimes m}\xi^\vw=\xi^\vw.$$
As we already showed in the proof of Proposition~\ref{P.G}, this is exactly the defining relation for the group $\Aut K$.

To show the opposite inclusion, denote by $\tilde u=\tilde\gamma\tilde v$ the fundamental representation of $\hat\Gamma\rtimes\Aut K$. We need to show that $\tilde u$ satisfies the relations corresponding to the intertwiners from~$\Cat^G$. Thanks to Proposition~\ref{P.ThatLB2}, we know the linear bases of the morphism spaces and we do not have to consider every possible $\HG\in\Cat$. So, we need to show that $\tilde u\xi^{\aw}=\xi^\aw$ for every $\aw\in V(K)^*$ such that $g_\aw\in F(K)$, where $\xi^\aw=\hat T^G_{(K,\emptyset,\aw)}$.

From Proposition~\ref{P.G}, we know that $\xi^\aw$ is an intertwiner of $\Aut K$, so we have
$$\tilde u^{\otimes k}\xi^\aw=\tilde\gamma^{\otimes k}\tilde v^{\otimes k}\xi^\aw=\tilde\gamma^{\otimes k}\xi^\aw,$$
where $k$ denotes the length of~$\aw$. It remains to show that $\tilde\gamma^{\otimes k}\xi^\aw=\xi^\aw$. This relation, rewritten in the tensor entries, reads
$$\tilde\gamma_{i_1}\cdots\tilde\gamma_{i_k}=1\quad\text{whenever $\exists\phi\colon K\to G$ inj.\ hom.\ such that $\phi(\aw)=\iw$.}$$
According to Lemma~\ref{L.greatestK}, the embedding of $K$ to $G$ is determined uniquely up to automorphisms of $K$. So, we must have $\phi\in\Aut K$ for the injective homomorphism above. So, we need to show that $\tilde\gamma_{\phi(a_1)}\cdots\tilde\gamma_{\phi(a_k)}=1$. Since $F(K)$ is supposed to be $\Aut K$-invariant, we have $\phi(g_\aw)\in F(K)$ and hence the mentioned relation is indeed satisfied in $\Z_2^{*V}/F(K)$.
\end{proof}

\section{Associating categories to semidirect product CMQGs}
\label{sec.conv}

In the previous text, we managed to determine the quantum groups associated to group-theoretical graph categories and skew categories. Namely it turned out that the quantum group is always of the form $\GGr=\hat\Gamma\rtimes\Aut K$. In this section, we would like to comment on the opposite direction: Given a~quantum group of such a~structure, what categories can we associate to it? What are the corresponding intertwiner spaces?

\subsection{Graph category associated to quantum group}
\begin{prop}
\label{P.conv}
Let $G$ be a graph and let $\Gamma$ be some $\Aut G$-invariant quotient of $\Z_2^{*V(G)}$. Then there exists a~skew graph category~$\Cat$ such that $\Cat^G$ forms the representation category associated to the quantum group $\GGr=\hat\Gamma\rtimes\Aut G$.
\end{prop}
\begin{proof}
Let $A$ be the $\Aut G$-invariant normal subgroup of $\Z_2^{*V(G)}$ such that $\Gamma=\Z_2^{*V(G)}/A$. We construct~$\Cat$ to be the smallest skew category containing all the bilabelled graphs $(G,\emptyset,\aw)$ with $g_\aw\in A$.

Recall from Proposition \ref{P.skewgen} that we are able to explicitly construct the associated graph fibration~$F$: The fibres of~$F$ are formed by the closure of the set~$\{N_0,N_1,G\}$ under arbitrary $f$-unions. Each fibre is then defined by $F(H):=\{\iota_H(a)\mid a\in A;\;\iota_H\colon G\to H\ppen\;\text{inj.\ hom.}\}$. In particular, we have $F(G)=A$, so $\Gamma=\Z_2^{*V(G)}/F(G)$. Therefore, the quantum group corresponding to the above defined category is indeed $\GGr=\hat\Gamma\rtimes\Aut G$.
\end{proof}

\begin{rem}
\label{R.converse}
By choosing $G:=N_n$ the edgeless graph on $n$ vertices, we get the following statement:

\textit{Let $A_n$ be an $S_n$-invariant normal subgroup of $\Z_2^{*n}$. Denote $\Gamma:=\Z_2^{*n}/A_n$. Then there exists a skew graph category $\Cat$ containing edgeless graphs only such that $\Cat^G$ forms the representation category associated to the quantum group $\hat\Gamma\rtimes S_n$.}

This statement does not hold for skew categories of partitions as they were defined in \cite{Maa18}. That is, although the category $\Cat$ from the statement above contains edgeless graphs only, it might not correspond to any skew category of partitions, because it might not be closed under removing isolated vertices (see Example~\ref{E.new}). In the group language, there need not exist an $S_\infty$-invariant normal subgroup $A\trianglelefteq\Z_2^{*\infty}$ such that $A_n$ is a subgroup of $A$ generated by $n$ of the canonical generators. Compare with \cite[Theorem~4.20]{Maa18}.
\end{rem}

\begin{rem}
The above mentioned proposition can be easily generalized replacing $\Aut G$ by $\Aut K$, but $K$ is not allowed arbitrary. It must hold that if $F$ is the closure of $\{N_0,N_1,K\}$ under arbitrary $f$-unions, then $K$ must be the greatest subgraph of $G$ contained in $F$. This condition can be equivalently formulated as follows: $K$~is a subgraph of~$G$ with $|V(K)|=|V(G)|$ and, for every $\phi\in\Aut G$, if $\phi(K)$ is a subgraph of $G$ (taking the same embedding), then $\phi\in\Aut K$.
\end{rem}

\begin{rem}
The graph category corresponding to a~given quantum group $\GGr=\hat\Gamma\rtimes\Aut G$ is not determined uniquely.  For example, the group $\GGr=\Aut G$ is known to be associated with the category of all graphs (which is also a~special case of Theorems \ref{T.Geasy},~\ref{T.G}). However, the construction described in the proof above would yield a~different category: The fibres of the graph fibration would be determined as the completion of~$\{N_0,N_1,G\}$ under $f$-unions. Consequently, it would not be a~full graph fibration and it would not correspond to the category of all graphs.

We could also modify the proof of Proposition~\ref{P.conv} and construct a full graph fibration $F$ corresponding to $\GGr=\hat\Gamma\rtimes\Aut G$. Every graph $H$ would form a fibre of $F$ determined by exactly the same formula $F(H)=\{\iota_H(a)\mid a\in A;\ppen\;\iota_H\colon G\to H\ppen\;\text{inj.\ hom.}\}$ and $F(H)=E$ trivial if $H$ is not a subgraph of $G$. Such a collection would be obviously closed under injective graph homomorphisms, so it would indeed form a full graph fibration. Returning back to the example $\GGr=\Aut G$, this construction would lead to yet another graph category. 
\end{rem}

We can formulate an analogous statement also in the easy case. Let us denote by $\End G$ the monoid of graph homomorphisms $G\to G$.
\begin{prop}
\label{P.easyconv}
Consider a graph $G$ and let $\Gamma$ be some $\End G$-invariant quotient of~$\Z_2^{*V(G)}$. Then there exists a~group-theoretical graph category~$\Cat$ such that $\Cat^G$ forms the representation category associated to the quantum group $\GGr=\hat\Gamma\rtimes\Aut G$.
\end{prop}
\begin{proof}
Let $A$ be the $\End G$-invariant normal subgroup of $\Z_2^{*V(G)}$ such that $\Gamma=\Z_2^{*V(G)}/A$. We construct an easy full graph fibration by putting $F(H):=\{\phi(a)\mid a\in A;\ppen\;\phi\colon G\to H\ppen\;\text{homomorphism}\}$ (cf.\ Prop.~\ref{P.grpthgen}) and $F(H):=E$ (the trivial group) if there is no homomorphism $G\to H$. Such a collection is certainly closed under graph homomorphisms, so according to Proposition~\ref{P.alteasy} it must be a full easy graph fibration.

Since $A$ is $\End G$-invariant, we have that $F(G)=A$, so $\Gamma=\Z_2^{*V(G)}/F(G)$ and hence the above defined graph fibration corresponds to the quantum group $\GGr=\hat\Gamma\rtimes\Aut G$.
\end{proof}

However, even in the easy case, the category is not determined uniquely. Take, for example, the graph~$G=\Gedge$ of two connected vertices, so $\Aut G=S_2$ and take $\Gamma=\Z_2\times\Z_2$. So we are looking for a~group-theoretical graph category~$\Cat$ such that $\Cat^G$ corresponds to $\GGr:=\hat\Gamma\rtimes\Aut G=\widehat{(\Z_2\times\Z_2)}\rtimes S_2=H_2$ -- the hyperoctahedral group. The following three categories are mutually distinct, but all of them correspond to~$\GGr$:
$$
\langle\Ggrpth,\Gabab\rangle,\qquad
\langle\Ggrpth,
\Graph{
\GS 1/0:2/1,2/3,3/2,3/4;
\GV 1:2,3;
}\rangle,\qquad
\langle\Ggrpth,
\Graph{
\GE 1/1:1/2;
\GS 1/0:1/0.7,1/1.3,2/1.7,2/2.3;
\GV 1:1,2;
},
\Graph{
\GS 1/0:2/1,2/3,3/2,3/4;
\GV 1:2,3;
}\rangle.$$
The first one is the category constructed in the proof of Proposition~\ref{P.easyconv} -- the smallest full category. We have encountered it already in Examples~\ref{E.abab}, \ref{E.abab2},~\ref{E.abab3}. The second one is the group-theoretical partition category corresponding to $\hat\Gamma\rtimes S_2$. The last one is generated by the previous two.

\subsection{Computing intertwiners}

In this subsection, we look on the problem of associating the intertwiner spaces
$$\RCat_\GGr(k,l)=\Mor(u^{\otimes k},u^{\otimes l})=\{T\colon(\C^n)^{\otimes k}\to(\C^n)^{\otimes l}\mid Tu^{\otimes k}=u^{\otimes l}T\}$$
to a~given quantum group $\GGr=(O(\GGr),u)$ of the form $\GGr=\hat\Gamma\rtimes\Aut G$.

First of all, let us mention what is, in our opinion, a~serious disadvantage of graph categories in comparison with categories of partitions: The morphism spaces $\Cat(k,l)$ of a~given graph category are typically infinite. This also means that the functor $\KG\mapsto T^G_\KG$ is highly non-injective.

As a~consequence, if we want to construct a~representation category $\Cat^G(k,l):=\spanlin\{T^G_{\KG}\mid \KG\in\Cat(k,l)\}$ (interpreting it afterwards as a~representation category~$\RCat_\GGr$ of some quantum group~$\GGr$), then it might be highly non-trivial to explicitly compute what those spaces actually are (for some given $k,l$) since we are making a~linear span over an infinite number of elements. In addition, in many applications, we need linear independence, that is, we really need to know a~basis of the intertwiner spaces (see e.g.\ \cite{Fre19survey,Jun19}).

Fortunately, this issue was solved in case of group-theoretical and skew graph categories in Section~\ref{secc.That}: Considering the quantum group $\GGr=\hat\Gamma\rtimes\Aut G$, we know that the basis of $\RCat_\GGr(k,l)$ is exactly the set $\{\hat T^{\aw\bw}_G\}_{[\aw,\bw]\in W_G(k,l)}$ (recall Prop.~\ref{P.ThatLB}). In the following section, we formulate a generalization of such statement replacing the group $\Aut G$ with any permutation group.

\begin{rem}
This disadvantage of graph categories is in general unavoidable. In particular, it is not possible to give such a formula for the intertwiner spaces for the so-called quantum automorphism group of graphs (determined by the category of all planar bilabelled graphs) as this would solve the quantum isomorphism problem, which was proven to be undecidable~\cite{AMR19}.
\end{rem}

\subsection{Intertwiners corresponding to any permutation group}
In this section, we generalize the results of Section~\ref{sec.QG} replacing the automorphism group $\Aut G$ or $\Aut K$ by any permutation group $H\subset S_n$. However, in this case we do not have such a diagrammatic calculus (at least not such a simple one given by bilabelled graphs) since the representation category of $H$ may not be generated by a single matrix $A\colon\C^n\to\C^n$.

\begin{defn}
Consider $H\subset S_n$. Considering $\aw\in\{1,\dots,n\}^k$, $\bw\in\{1,\dots,n\}^l$, we denote by $\hat T_H^{\aw\bw}\colon (\C^n)^{\otimes k}\to(\C^n)^{\otimes l}$ the linear mapping with entries
$$[\hat T_H^{\aw\bw}]_{\jw\iw}=\#\{\phi\in H\mid\phi(\aw)=\iw,\phi(\bw)=\jw\}.$$
\end{defn}

Note that given a graph $G$, we have $\hat T^{\aw\bw}_{\Aut G}=\hat T^{\aw\bw}_G=\hat T^G_{(G,\aw,\bw)}$. Moreover, we have the following.

\begin{lem}
\label{L.THpart}
Fix $n\in\N$ and $H\subset S_n$. For any partition $\PG\in\Part(k,l)$, we have
$$\hat T_\PG^{(n)}=\frac{1}{|H|}\sum_{\substack{\aw,\bw\in\{1,\dots,n\}^*\\\PG=\ker(\aw,\bw)}}\hat T_H^{\aw\bw}$$
\end{lem}
\begin{proof}
Let us look on the entries of the right-hand side
$$\sum_{\substack{\aw,\bw\in\{1,\dots,n\}^*\\\PG=\ker(\aw,\bw)}}[\hat T_H^{\aw\bw}]_{\jw\iw}=\sum_{\substack{\aw,\bw\in\{1,\dots,n\}^*\\\PG=\ker(\aw,\bw)}}\#\{\phi\in H\mid \phi(\aw,\bw)=(\iw,\jw)\}.$$
Assuming that $\phi(\aw,\bw)=(\iw,\jw)$, then obviously $\PG=\ker(\aw,\bw)$ if and only if $\PG=\ker(\iw,\jw)$. So, the above equals to zero whenever $\PG\neq\ker(\iw,\jw)$. On the other hand, assume that $\PG=\ker(\iw,\jw)$. Choose two pairs $(\aw_1,\bw_1)$ and $(\aw_2,\bw_2)$ such that $\PG=\ker(\aw_k,\bw_k)$, $k=1,2$. Choose also two permutations $\phi_1,\phi_2\in H$ such that $\phi_k(\aw_k,\bw_k)=(\iw,\jw)$. Then, if the pairs $(\aw_k,\bw_k)$ are distinct, the permutations $\phi_k$ must also be distinct. Consequently, the sum above counts every permutation in~$H$ exactly once, so the result must be equal to $|H|$.
\end{proof}

\begin{prop}
Consider $H\in S_n$ and $\aw,\bw,\cw,\dw\in\{1,\dots,n\}^*$. It holds that
\begin{enumerate}
\item $\displaystyle\hat T_H^{\aw\bw}\otimes\hat T_H^{\cw\dw}=\sum_{\eta\in H}T_H^{\aw\eta(\cw)\,\bw\eta(\dw)}$,
\item $\displaystyle\hat T_H^{\aw\bw}\hat T_H^{\cw\dw}=\sum_{\substack{\eta\in H\\\bw=\eta(\cw)}}T_H^{\aw\eta(\dw)}$ (assuming $\bw$ and $\cw$ have the same length),
\item $\displaystyle (\hat T_H^{\aw\bw})^*=T_H^{\bw\aw}$.
\end{enumerate}
In particular, the collection $\Cat^H(k,l):=\spanlin\{T_H^{\aw,\bw}\mid \aw\in\{1,\dots,n\},\bw\in\{1,\dots,n\}\}$ is a representation category
\end{prop}
\begin{proof}
For the tensor product, we have by definition
$$[\hat T^{\aw\bw}_H\otimes\hat T^{\cw\dw}_H]_{\mathbf{ik\,jl}}=\#\{(\phi,\psi)\in H^2\mid \phi(\aw,\bw)=(\iw,\jw),\;\psi(\cw,\dw)=(\mathbf{k},\mathbf{l})\}.$$
Now, given such a pair $(\phi,\psi)$, we can denote by $\eta\in H$ their ``difference'', that is, $\eta:=\phi^{-1}\circ\psi$. Consequently, the above equals to
$$=\sum_{\eta\in H}\#\{\phi\in H\mid \phi(\aw,\bw)=(\iw,\jw),\;(\phi\circ\eta)(\cw,\dw)=(\mathbf{k},\mathbf{l})\}=\sum_{\eta\in H}[T_H^{\aw\eta(\cw)\,\bw\eta(\dw)}]_{\iw\mathbf{k}\,\jw\mathbf{l}}.$$

The proof of (2) goes similar way:
\begin{align*}
[\hat T^{\aw\bw}_H\hat T^{\cw\dw}_H]_{\iw\jw}
&=\sum_{\mathbf{k}}\#\{(\phi,\psi)\in H^2\mid \phi(\aw)=\iw,\;\phi(\bw)=\mathbf{k}=\psi(\cw),\;\psi(\dw)=\jw\}\\
&=\sum_{\eta\in H}\#\{\phi\in H\mid\phi(\aw)=\iw,\;\phi(\bw)=(\phi\circ\eta)(\cw),\;(\phi\circ\eta)(\dw)=\jw\}\\
&=\sum_{\substack{\eta\in H\\\bw=\eta(\cw)}}[T^{\aw\eta(\dw)}]_{\iw\jw}.
\end{align*}

Item (3) is clear directly from the definition.

From Lemma~\ref{L.THpart}, it follows that $\Cat^H$ contains the identity morphism $\hat T^{(n)}_{\Gid}$ as well as the duality morphism $\hat T^{(n)}_{\Gpair}$. Since it is also closed under tensor products, compositions and the involution, it must form a representation category.
\end{proof}

We again denote by $W_H(k):=\{1,\dots,n\}^k/H$ the set of equivalence classes of the tuples $\aw\in\{1,\dots,n\}^k$ with respect to the action of $H$. We denote the equivalence classes by $[\aw]$. Similarly, we denote by $W_H(k,l)$ the equivalence classes $[\aw,\bw]$ of pairs of tuples. Similarly to the operators associated to bilabelled graphs, the operators $\hat T^{\aw\bw}$ do not depend on the particular representative of the class $[\aw,\bw]$.

\begin{prop}
\label{P.intsp}
Let $H\subset S_n$ be a permutation group. Then $\RCat_H=\Cat^H$. That is, denoting by~$u$ the fundamental representation of $H$, we have
$$\Mor(u^{\otimes k},u^{\otimes l})=\spanlin\{\hat T^{\aw\bw}_H\mid [\aw,\bw]\in W_H(k,l)\}.$$
The sets of linear maps $\{\hat T^{\aw\bw}_H\}_{[\aw,\bw]\in W_H(k,l)}$ are linearly independent.
\end{prop}
\begin{proof}
The linear independence is proven the same way as in Lemma~\ref{L.ThatLI}. It follows from the fact that $[\hat T^{\aw\bw}_H]_{\jw\iw}[\hat T^{\aw'\bw'}_H]_{\jw\iw}=0$ unless $[\aw,\bw]=[\aw',\bw']$ in $W_H(k,l)$.

To show that $\RCat_H=\Cat^H$, we can essentially copy the proof of Proposition~\ref{P.G} replacing $\Aut K$ by $H$. To prove the inclusion $\supset$, take any $\sigma\in H$ and show that the associated permutation matrix $A_\sigma$ satisfy the intertwiner relations $A_\sigma\xi^\aw=\xi^\aw$ with $\xi^\aw=\hat T_H^{\emptyset\aw}$. This is equivalent with saying that
$$\#\{\phi\in H\mid \phi(\aw)=\iw\}=\#\{\phi\in H\mid (\sigma\circ\phi)(\aw)=\iw\},$$
which is obviously true.

To prove the opposite inclusion, it is enough to check that taking any $\sigma\in S_n$ satisfying the intertwiner relations of $\Cat^H$, we necessarily have $\sigma\in H$. We get this by choosing $\aw=\iw=(1,\dots,n)$ since then
\begin{align*}
[\xi^\aw]_\aw&=\{\phi\in H\mid\phi=\id\}=1\\
[A_\sigma\xi^\aw]_\aw&=\{\phi\in H\mid \sigma\circ\phi=\id\}=
\begin{cases}
1&\text{if $\sigma^{-1}\in H$}\\
0&\text{otherwise}
\end{cases}\qedhere
\end{align*}
\end{proof}

\begin{thm}
\label{T.intsp}
Let $H\subset S_n$ be a permutation group and consider $\Gamma=\Z_2^{*n}/A$ with $A\trianglelefteq\Z_2^{*n}$ being $H$-invariant. Denote by~$u$ the fundamental representation of the quantum group $\GGr:=\hat\Gamma\rtimes H$. Then
$$\Mor(u^{\otimes k},u^{\otimes l})=\spanlin\{\hat T^{\aw\bw}_H\mid [\aw,\bw]\in W_H(k,l)\text{ such that }g_{\aw\bw^*}\in A\}.$$
The sets of linear maps $\{\hat T^{\aw\bw}_H\}_{[\aw,\bw]\in W_H(k,l)}$ are linearly independent.
\end{thm}
\begin{proof}
We denote by $\GGr$ the quantum group corresponding to the representation category $\Cat^H$ and show that $\GGr=\hat\Gamma\rtimes H$. The proof is basically the same as in case of Theorem \ref{T.G}. In particular, thanks to Lemma~\ref{L.THpart}, we can do the ``partition tricks'' also here. For instance, we immediatelly know that $\GGr\subset\hat\Z_2^{*n}\rtimes S_n$ since $\hat T_{\Ggrpth}^{(n)}\in\Cat^H$. The linear independence was proven in Proposition~\ref{P.intsp}.
\end{proof}

\bibliographystyle{halpha}
\bibliography{mybase}

\end{document}